\documentclass{aims}
\usepackage{amsmath}
  \usepackage{paralist}
  \usepackage{graphics} 
  \usepackage{epsfig} 
 \usepackage[colorlinks=true]{hyperref}

\allowdisplaybreaks[4]

\usepackage{graphicx,amssymb,mathrsfs,latexsym,amsthm}
\usepackage{verbatim}

\usepackage{color,soul,framed,xcolor}

\definecolor{shadecolor}{rgb}{1,1,0}

\hypersetup{urlcolor=blue, citecolor=red}

    \def\ppages{X--XX}

\newtheorem{theorem}{Theorem}[section]

\newtheorem{lemma}{Lemma}[section]
\newtheorem{proposition}{Proposition}[section]

\theoremstyle{definition}
\newtheorem{definition}{Definition}[section]
\newtheorem{remark}{Remark}[section]

\title[Relative tail entropy]
      {Relative tail entropy for random bundle transformations}

\author[Xianfeng Ma, Ercai Chen]{}


 \email{xianfengma@gmail.com}
 \email{ecchen@njnu.edu.cn}

\thanks{The second and third authors are  supported by  National Natural
Science Foundation of China (11271191).
The  second author is supported by  National Natural
Science Foundation of China (11471114).
The third author is partially supported by National
Basic Research Program of China (973 Program) (grant number 2013CB834100).
}

\begin{document}
\maketitle

\renewcommand{\thefootnote}{\fnsymbol{footnote}}
\centerline{
\scshape Xianfeng Ma
}
\medskip
{\footnotesize
   \centerline{Department of Mathematics}
   \centerline{East China University of Science and Technology, Shanghai 200237, China}
} 

\medskip

\centerline{\scshape Ercai Chen}
\medskip
{\footnotesize
 \centerline{School of Mathematical Science}
   \centerline{Nanjing Normal University, Nanjing 210097, China}
   \centerline{and}
   \centerline{Center of Nonlinear Science}
   \centerline{Nanjing University, Nanjing 210093, China}
} %

\bigskip


\begin{abstract}
We introduce the relative tail entropy to establish a variational principle for continuous bundle random dynamical systems.
We also show that the relative tail entropy is conserved by the principal extension.
\end{abstract}

\section{Introduction}\label{sec1}

The entropy measures the complexity of a dynamical systems  both in the topological and measure-theoretic settings.
The topological entropy measures the maximal dynamical complexity versus an average complexity reflected by the measure-theoretic entropy.
The relationship between these two kinds of entropy is the classical variational principle, which states that the topological entropy is the supremum of the measure-theoretic entropy over all invariant measures \cite{Goodman, Goodwyn,Misiurewicz1}.

The entropy concepts can be localized by defining topological tail entropy to quantify the amount of disorder or uncertainty in a system at arbitrary small scales \cite{Misiurewicz}.
The local complexity of a dynamical system can also be measured by the defect of uniformity in the convergence of the measure-theoretic entropy function.
A variational principle related these two aspects is established  in the case of homeomorphism from subtle results in the theory of entropy structure by Downarowicz \cite{Down2005,Boyle}.
An elementary proof of this variational principle for continuous transformations is obtained in terms of essential partitions by Burguet \cite{Burguet}.
Ledrappier \cite{Ledrappier} presents a variational principle between the topological tail entropy and the defect of upper semi-continuity of the measure-theoretic entropy on the cartesian square of the dynamical system, and prove that topological tail entropy is an invariant under any principal extension.

Kifer and Weiss \cite{Kifer2002} introduce the relative tail entropies for continuous bundle RDSs by investigating the open covers and spanning subsets and deduce the equivalence between the two notions. It is shown in \cite{KiferLiu2006} that the defects of the upper semi-continuity of the relative measure-theoretic entropies are bounded from above by the relative tail entropy.

In this paper
we devote to proposing a relative variational principle for the relative tail entropy introduced by using open random covers, which enable us to treat different fibers with different open covers.
We also introduce the factor transformation and consider its basic properties related to the invariant measure and the upper semi-continuity of the relative measure-theoretic entropy for continuous bundle RDSs.
For the product RDS generated by a given  RDS and any other  RDS with the same probability space, we obtain a variational inequality, which shows that the defect of the upper semi-continuity of the relative measure-theoretic entropy of any invariant measure in the product RDS cannot exceed the relative tail entropy of the original RDS.
When the two continuous bundle RDSs coincide, we construct a maximal invariant measure to ensure that the relative tail entropy could be reached, and establish the variational principle.
For the probability space being trivial, it reduces to the variational principle deduced by Ledrappier  \cite{Ledrappier} in deterministic dynamical systems.
As an application of the variational principle we show that the relative tail entropy is conserved by any principal extension.

The paper is organized as follows.
In Section \ref{sec2}, we recall some background in the ergodic theory, introduce the relative tail entropy with respect to open random covers and state our main results.
In Section \ref{sec3}, we give some basic properties of the relative entropy and the relative tail entropy.
In Section \ref{sec4}, we devote to the proof of the variational principle and show that the relative tail entropy is  an invariant under principal extensions.

\section{Preliminaries and main results}\label{sec2}

Let $(\Omega, \mathcal{F},\mathbb{P})$ be a complete countably generated probability space together with a $\mathbb{P}$-preserving transformation $\vartheta$ and $(X,\mathcal{B})$ be a compact metric space with the Borel $\sigma$-algebra $\mathcal{B}$.
Let $\mathcal{E}$ be a measurable subset of $\Omega\times X$ with respect to the product $\sigma$-algebra $\mathcal{F}\times \mathcal{B}$ and the fibers $\mathcal{E}_{\omega}=\{x\in X: (\omega,x)\in \mathcal{E}\}$ be compact.
A continuous bundle random dynamical system (RDS) $T$ over $(\Omega, \mathcal{F},\mathbb{P},\vartheta)$ is generated by the mappings $T_{\omega}:\mathcal{E}_{\omega}\rightarrow \mathcal{E}_{\vartheta\omega}$
so that the map $(\omega,x)\rightarrow T_{\omega}x$ is measurable and the map $x\rightarrow T_{\omega}x$ is continuous for $\mathbb{P}$-almost all (a.a.) $\omega$.
The family $\{T_{\omega}:\omega\in \Omega\}$ is called a random transformation and each $T_{\omega}$ maps the fiber $\mathcal{E}_{\omega}$ to $\mathcal{E}_{\vartheta\omega}$.
The map $\Theta:\mathcal{E}\rightarrow \mathcal{E}$ defined by $\Theta(\omega,x)=(\vartheta\omega, T_{\omega}x)$ is called the skew product transformation. Observe that $\Theta^n(\omega, x)=(\vartheta^n\omega, T_{\omega}^nx)$, where $T_{\omega}^n=T_{\vartheta^{n-1}\omega}\circ\cdots T_{\vartheta\omega}\circ T_{\omega}$ for $n\geq 0$ and $T_{\omega}^0=id$.

Let $\mathcal{P}_{\mathbb{P}}(\Omega\times X)$ be the space of probability measures on $\Omega\times X$ having the marginal $\mathbb{P}$ on $\Omega$
and set $\mathcal{P}_{\mathbb{P}}(\mathcal{E})=\{\mu \in \mathcal{P}_{\mathbb{P}}(\Omega\times X): \mu(\mathcal{E})=1\}$.
Denote by $\mathcal{I}_{\mathbb{P}}(\mathcal{E})$ the space of all $\Theta-$invariant measures in $\mathcal{P}_{\mathbb{P}}(\mathcal{E})$.

Let $\mathcal{S}$ be a sub-$\sigma$-algebra of $\mathcal{F}\times \mathcal{B}$ restricted on $\mathcal{E}$, $\mathcal{R}=\{R_i\}$ be a finite or countable partition of $\mathcal{E}$ into measurable sets. For $\mu\in \mathcal{P}_{\mathbb{P}}(\Omega\times X)$ the conditional entropy of $\mathcal{R}$ given $\sigma$-algebra $\mathcal{S}$ is defined as
\begin{equation*}
H_{\mu}(\mathcal{R}\mid \mathcal{S})=-\int \sum_{i}E(1_{R_i}\mid \mathcal{S}) \log E(1_{R_i}\mid \mathcal{S})d\mu,
\end{equation*}
where $E(1_{R_i}\mid \mathcal{S})$ is the conditional expectation of $1_{R_i}$ with respect to $\mathcal{S}$.

Let $\mu\in \mathcal{I}_{\mathbb{P}}(\mathcal{E})$ and $\mathcal{S}$ is a sub$-\sigma-$algebra of $\mathcal{F}\times \mathcal{B} $ restricted on $\mathcal{E}$ satisfying $\Theta^{-1}\mathcal{S}\subset \mathcal{S}$. For a given measurable partition $\mathcal{R}$ of $\mathcal{E}$, the conditional entropy
$H_{\mu}(\mathcal{R}^{(n)}\mid \mathcal{S})$ is a non-negative sub-additive sequence, where $\mathcal{R}^{(n)}=\bigvee_{i=0}^{n-1}(\Theta^i)^{-1}\mathcal{R}$.
The {\it relative entropy $h_{\mu}(\mathcal{R}\mid \mathcal{S})$ of $\Theta$ with respect to a partition $\mathcal{R}$ } is defined as
\begin{equation*}
h_{\mu}(\mathcal{R}\mid \mathcal{S})=\lim_{n\rightarrow\infty}\frac{1}{n}H_{\mu}(\mathcal{R}^{(n)}\mid \mathcal{S})=\inf_n\frac{1}{n}H_{\mu}(\mathcal{R}^{(n)}\mid \mathcal{S}).
\end{equation*}
The {\it relative entropy of $\Theta$ } is defined by the formula
\begin{equation*}
h_{\mu}(\Theta\mid \mathcal{S})=\sup_{\mathcal{R}}h_{\mu}(\mathcal{R}\mid \mathcal{S}),
\end{equation*}
where the supremum is taken over all finite or countable measurable partitions $\mathcal{R}$ of $\mathcal{E}$ with finite conditional entropy $H_{\mu}(\mathcal{R}\mid \mathcal{S})<\infty$.
The {\it defect of upper semi-continuity of the relative entropy $h_{\mu}(\Theta\mid \mathcal{S}$)} is defined on $\mathcal{I}_{\mathbb{P}}(\mathcal{E})$ as
\begin{equation*}
h^*_m(\Theta\mid \mathcal{S})=
\begin{cases}
\limsup\limits_{\mu\rightarrow m} h_{\mu}(\Theta\mid \mathcal{S})- h_m(\Theta\mid \mathcal{S}),
& \text{if}\,\, h_m(\Theta\mid \mathcal{S})<\infty,\\
\infty, & \text{otherwise}.
\end{cases}
\end{equation*}

Any $\mu \in \mathcal{P}_{\mathbb{P}}(\mathcal{E})$ on $\mathcal{E}$ disintegrates $d\mu(\omega,x)=d\mu_{\omega}(x)d\mathbb{P}(\omega)$ (see \cite[Section 10.2]{Dudley}), where $\mu_{\omega}$ are regular conditional probabilities with respect to the $\sigma-$algebra $\mathcal{F}_{\mathcal{E}}$ formed by all sets $(F\times X)\cap \mathcal{E}$ with $F\in \mathcal{F}$.
This means that $\mu_{\omega}$ is a probability measure on $\mathcal{E}_{\omega}$ for $\mathbb{P}$-a.a. $\omega$ and for any measurable set $R\in \mathcal{E}$,
$\mathbb{P}$-a.s. $\mu_{\omega}(R(\omega))=E(R\mid\mathcal{F}_{\mathcal{E}})$ , where $R(\omega)=\{x: (\omega,x)\in R\}$ and so $\mu(R)=\int \mu_{\omega}(R(\omega))d\mathbb{P}(\omega)$.
The conditional entropy of $\mathcal{R}$ given $\sigma-$algebra $\mathcal{F}_{\mathcal{E}}$ can be written as
\begin{equation*}
H_{\mu}(\mathcal{R}\mid \mathcal{F}_{\mathcal{E}})=-\int \sum_i E(R_i\mid\mathcal{F}_{\mathcal{E}})\log E (R_i\mid\mathcal{F}_{\mathcal{E}})d \mathbb{P}=\int H_{\mu_{\omega}}(\mathcal{R}(\omega))d\mathbb{P},
\end{equation*}
where $\mathcal{R}(\omega)=\{R_i(\omega)\}$, $R_i(\omega)=\{x\in\mathcal{E}_{\omega}:(\omega,x)\in R_i\}$ is a partition of $\mathcal{E}_{\omega}$.

Let $(Y, \mathcal{C})$ be a compact metric space with the Borel $\sigma$-algebra $\mathcal{C}$ and $\mathcal{G}$ be a measurable, with respect to the product $\sigma$-algebra $\mathcal{F}\times \mathcal{C}$, subset of $\Omega\times Y$ with the fibers $\mathcal{G}_{\omega}$ being compact.
The continuous bundle RDS $S$ over $(\Omega, \mathcal{F},\mathbb{P},\vartheta)$ is generated by the mappings $S_{\omega}:\mathcal{G}_{\omega}\rightarrow \mathcal{G}_{\vartheta\omega}$
so that the map $(\omega,y)\rightarrow S_{\omega}y$ is measurable and the map $y\rightarrow S_{\omega}y$ is continuous for $\mathbb{P}$-almost all (a.a.) $\omega$.
The skew product transformation $\Lambda:\mathcal{G}\rightarrow \mathcal{G}$ is defined as $\Lambda(\omega,y)=(\vartheta\omega, S_{\omega}y)$.

\begin{definition}
Let $T, S$ are two continuous bundle RDSs over $(\Omega,\mathcal{F},\mathbb{P},\vartheta)$ on $\mathcal{E}$ and $\mathcal{G}$, respectively.
$T$ is said to be a {\it factor } of $S$, or that $S$ is an {\it extension} of $T$,
if there exists a family of continuous surjective maps $\pi_{\omega}:\mathcal{G}_{\omega}\rightarrow \mathcal{E}_{\omega}$ such that the map $(\omega,y)\rightarrow \pi_{\omega}y$ is measurable and $\pi_{\vartheta\omega}S_{\omega}=T_{\omega}\pi_{\omega}$.
The map $\pi:\mathcal{G}\rightarrow\mathcal{E}$ defined by $\pi(\omega,y)=(\omega,\pi_{\omega}y)$ is called the {\it factor or extension transformation} from $\mathcal{G}$ to $\mathcal{E}$.
The skew product system $(\mathcal{E},\Theta)$ is  called a {\it factor} of $(\mathcal{G},\Lambda)$ or that $(\mathcal{G},\Lambda)$ is an {\it extension } of $(\mathcal{E},\Theta)$.
\end{definition}

Denote by $\mathcal{A}$ the restriction of $\mathcal{F}\times \mathcal{B}$ on $\mathcal{E}$
and set $\mathcal{A}_{\mathcal{G}}=\{ \pi^{-1} A : A\in \mathcal{A}\}$.

\begin{definition}
A continuous bundle RDS $T$ on $\mathcal{E}$ is called a {\it principal factor} of $S$ on $\mathcal{G}$, or that $S$ is a {\it principal extension} of $T$, if for any $\Lambda-$invariant probability measure $m$ in  $\mathcal{I}_{\mathbb{P}}(\mathcal{G})$, the relative entropy of $\Lambda$ with respect to $\mathcal{A}_{\mathcal{G}}$ vanishes, {\it i.e.}, $h_{\mu}(\Lambda\mid \mathcal{A}_{\mathcal{G}})=0.$
\end{definition}

Let $T$ and $S$ are two continuous bundle RDSs over $(\Omega,\mathcal{F},\mathbb{P},\vartheta)$ on $\mathcal{E}$ and $\mathcal{G}$, respectively.
Let $\mathcal{H}=\{(\omega,y,x): y\in \mathcal{G}_{\omega}, x\in \mathcal{E}_{\omega}\}$ and
$\mathcal{H}_{\omega}=\{(y,x):  (\omega,y,x)\in \mathcal{H}\}$. It is not hard to see that $\mathcal{H}$ is a measurable subset of $\Omega\times Y\times X$ with respect to the product $\sigma-$algebra $\mathcal{F}\times\mathcal{C}\times \mathcal{B}$ (as a graph of a  measurable multifunction; see \cite[Proposition III.13]{Castaing}).
The continuous bundle RDS $S\times T$ over $(\Omega, \mathcal{F},\mathbb{P},\vartheta)$ is generated by the family of mappings $(S\times T )_{\omega}:\mathcal{H}_{\omega}\rightarrow \mathcal{H}_{\vartheta\omega}$ with $(y,x)\rightarrow (S_{\omega}y, T_{\omega}x)$.
The map $(\omega,y,x)\rightarrow (S_{\omega}y, T_{\omega}x)$ is measurable and the map $(y,x)\rightarrow (S_{\omega}y, T_{\omega}x)$ is continuous in $(y,x)$ for $\mathbb{P}$-a.a. $\omega$.
The skew product transformation $\Gamma$ generated by $\Theta$ and $\Lambda$ from $\mathcal{H}$ to itself is defined as  $\Gamma (\omega,y,x) =(\vartheta\omega,S_{\omega}y, T_{\omega}x)$.

Let $\pi_{\mathcal{E}}:\mathcal{H} \rightarrow\mathcal{E}$ be the natural projection   with $\pi_{\mathcal{E}}(\omega,y,x)=(\omega,x)$, and $\pi_{\mathcal{G}}:\mathcal{H}\rightarrow \mathcal{G}$  with  $\pi_{\mathcal{G}}(\omega,y,x)=(\omega,y)$.
Then $\pi_{\mathcal{E}}$ and $\pi_{\mathcal{G}}$ are two factor transformations from $\mathcal{H}$ to $\mathcal{E}$ and $\mathcal{G}$, respectively.
Denote by $\mathcal{D}$ the restriction of $\mathcal{F}\times \mathcal{C}$ on $\mathcal{G}$
and set
$\mathcal{D}_{\mathcal{H}}=\pi^{-1}_{\mathcal{G}}( \mathcal{D})=\{(D\times X)\cap \mathcal{H}:D\in \mathcal{D}\}$,
$\mathcal{A}_{\mathcal{H}}=\pi^{-1}_{\mathcal{E}}( \mathcal{A})=\{(A\times Y)\cap \mathcal{H}:A\in \mathcal{A}\}$,
and $\mathcal{F}_{\mathcal{H}}=\{(F\times Y \times X )\cap \mathcal{H}: F\in \mathcal{F}   \}$.

The {\it relative entropy of $\Gamma$ given the $\sigma-$algebra $\mathcal{D}_{H}$} is defined by
\begin{equation*}
h_{\mu}(\Gamma\mid \mathcal{D}_{H})=\sup_{\mathcal{R}}h_{\mu}(\mathcal{R}\mid \mathcal{D}_{H}),
\end{equation*}
where
\begin{equation*}
h_{\mu}(\mathcal{R}\mid \mathcal{D}_{H})=\lim_{n\rightarrow\infty}
\frac{1}{n}H_{\mu}(\bigvee_{i=0}^{n-1}(\Gamma^i)^{-1}\mathcal{R}\mid \mathcal{D}_{H})
\end{equation*}
is the {\it relative entropy of $\Gamma $ with respect to a measurable partition $\mathcal{R}$},
and
the supremum is taken over all finite or countable measurable partitions $\mathcal{R}$ of $\mathcal{H}$ with finite conditional entropy $H_{\mu}(\mathcal{R}\mid \mathcal{D}_{H})<\infty$.

Let $\mathcal{E}^{(2)}=\{(\omega,x,y):x,y\in \mathcal{E}_{\omega}\}$, which is also a measurable subset of $\Omega\times X^2$ with respect to the product $\sigma-$algebra $\mathcal{F}\times \mathcal{B}^2$.
Let $\Theta^{(2)}:\mathcal{E}^{(2)}\rightarrow \mathcal{E}^{(2)}$ be a skew-product transformation with $\Theta^{(2)}(\omega,x,y)=(\vartheta\omega,T_{\omega}x, T_{\omega}y)$.
The map $(\omega,x,y)\rightarrow (T_{\omega}x, T_{\omega}y)$ is measurable and the map $(x,y)\rightarrow (T_{\omega}x, T_{\omega}y)$ is continuous in $(x,y)$ for $\mathbb{P}$-a.a. $\omega$.
Let $\mathcal{E}_1, \mathcal{E}_2$ be two copies of $\mathcal{E}$, {\it i.e.},  $\mathcal{E}_1=\mathcal{E}_2=\mathcal{E}$, and   $\pi_{\mathcal{E}_i}$  be the natural projection from $\mathcal{E}^{(2)}$ to $\mathcal{E}_i$ with  $\pi_{\mathcal{E}_i}(\omega,x_1,x_2)=(\omega,x_i)$, i=1, 2.
Denote by $\mathcal{A}_{\mathcal{E}^{(2)}}=\{(A\times X)\cap\mathcal{E}^{(2)}: A\in \mathcal{F}\times \mathcal{B}\}$.

The {\it relative entropy of $\Theta^{(2)}$ given the $\sigma-$algebra $\mathcal{A}_{\mathcal{E}^{(2)}}$} is defined by
\begin{equation*}
h_{\mu}(\Theta^{(2)}\mid \mathcal{A}_{\mathcal{E}^{(2)}})=\sup_{\mathcal{R}}h_{\mu}(\mathcal{R}\mid \mathcal{A}_{\mathcal{E}^{(2)}}),
\end{equation*}
where
\begin{equation*}
h_{\mu}(\mathcal{R}\mid \mathcal{A}_{\mathcal{E}^{(2)}})=\lim_{n\rightarrow\infty}
\frac{1}{n}H_{\mu}(\bigvee_{i=0}^{n-1}((\Theta^{(2)})^i)^{-1}\mathcal{R}\mid \mathcal{A}_{\mathcal{E}^{(2)}})
\end{equation*}
is the {\it relative entropy of $\Theta^{(2)} $ with respect to a measurable partition $\mathcal{R}$},
and
the supremum is taken over all finite or countable measurable partitions $\mathcal{R}$ of $\mathcal{E}^{(2)}$ with finite conditional entropy $H_{\mu}(\mathcal{R}\mid \mathcal{A}_{\mathcal{E}^{(2)}})<\infty$.

A (closed) random set $Q$ is a  measurable set valued map $\mathcal{Q}:\Omega\rightarrow 2^X $, or the graph of  $Q$ denoted by the same letter, taking values in the (closed)  subsets of compact metric space $X$.
An open random set $U$ is a  set valued map $U:\Omega\rightarrow 2^X $ whose complement $U^c$ is a closed random  set.
A measurable set $Q$ is an open (closed) random  set if the fiber $Q_{\omega}$ is an open (closed) subset of $\mathcal{E}_{\omega}$ in its induced topology from $X$ for $\mathbb{P}-$almost all $\omega$(see \cite[Lemma 2.7]{Crauel}).
A random cover $\mathcal{Q}$ of $\mathcal{E}$ is a finite or countable family of random sets $\{Q\} $  such that $\mathcal{E}_{\omega}=\bigcup_{Q\in \mathcal{Q}}Q(\omega)$ for all $\omega \in \Omega$, and it will be called an open random cover if all $Q\in \mathcal{Q}$ are open random sets.
Set $\mathcal{Q}(\omega)=\{Q(\omega)\}$, $\mathcal{Q}^{(n)}=\bigvee_{i=0}^{n-1}(\Theta^i)^{-1}\mathcal{Q}$ and  $\mathcal{Q}^{(n)}(\omega)=\bigvee_{i=0}^{n-1}(T_{\omega}^i)^{-1}\mathcal{Q}(\vartheta^i\omega)$.
Denote by $\mathfrak{P}(\mathcal{E})$ the set of random covers and $\mathfrak{U}(\mathcal{E})$ the set of open random covers. For $\mathcal{R}, \mathcal{Q}\in \mathfrak{P}(\mathcal{E})$, $\mathcal{R}$ is said  to be finer than $\mathcal{Q}$, which we will write $\mathcal{R}\succ\mathcal{Q}$ if each element of $\mathcal{R}$ is contained in some element of $\mathcal{Q}$.

For any non-empty set $S\subset \mathcal{E}$ and a random cover $\mathcal{R}\in \mathfrak{P}(\mathcal{E})$,
 let $N(S,\mathcal{R})=\min\{\text{card}(\mathcal{U}):\mathcal{U}\subset \mathcal{R}, S\subset \cup_{u\in \mathcal{U}} U\}$ and $N(\emptyset,\mathcal{R} )=1$.
Denote by  $N(S,\mathcal{R})(\omega)=\min\{\text{card}(\mathcal{U}(\omega)):\mathcal{U}(\omega)\subset \mathcal{R}(\omega), S(\omega)\subset \cup_{u\in \mathcal{U}} U(\omega)\}$.
Clearly, $N(S,\mathcal{R})(\omega)\leq N(S,\mathcal{R})$ for each $\omega$.
For $\mathcal{R},\mathcal{Q}\in \mathfrak{P}(\mathcal{E})$, let
$N(\mathcal{R}\mid\mathcal{Q})=\max_{Q\in\mathcal{Q}}N(Q, \mathcal{R})$
and
$N(\mathcal{R}\mid\mathcal{Q})(\omega)=\max_{Q\in\mathcal{Q}}N(Q, \mathcal{R})(\omega)$.

\begin{lemma}
Let $\mathcal{R}\in \mathfrak{U}(\mathcal{E})$ and  $ \mathcal{Q}\in \mathfrak{P}(\mathcal{E})$. The function $\omega\rightarrow N(\mathcal{R}\mid \mathcal{Q})(\omega)$ is measurable.
\end{lemma}

\begin{proof}
Let $Q\in \mathcal{Q}$ and $\mathcal{R}=\{R_1,\dots,R_l\}$.
For each $\omega,$ there exists a subset $\{j_1,\dots j_k\}$ of $\{1,\dots,l\}$ such that
$Q(\omega)\subset \cup_{i=1}^k R_{j_i}(\omega)$.
Let
\begin{equation*}
\Omega_{j_1,\dots,j_k}=\{ \omega\in \Omega: Q(\omega)\subset \bigcup_{i=1}^k R_{j_i}(\omega)\}.
\end{equation*}
Since $Q$ is a random set and
\begin{equation*}
\Omega_{j_1,\dots,j_k}= \Omega\setminus \{\omega: \big(\mathcal{E}\setminus \bigcup_{i=1}^k R_{j_i}(\omega) \big) \cap Q(\omega)\neq \emptyset     \},
\end{equation*}
$\Omega_{j_1,\dots,j_k}$ is a measurable subset of $\Omega$
(see for instance \cite[Theorem II.30]{Castaing}).
One obtain a finite partition of $\Omega$ into measurable sets $\Omega^J$, where $J$ is a finite family of subsets of $\{1,\dots,l\}$ such that
$\Omega^J=\bigcap_{(j_1,\dots,j_k)\in J}\Omega_{j_1,\dots,j_k}$.
Thus
for each $\omega$,
\begin{equation*}
N(Q,\mathcal{R})(\omega)=\min_{(j_1,\dots,j_k)\in J, 1\leq k\leq l} \text{card}\{j_1,\dots,j_k\},
\end{equation*}
and $N(Q,\mathcal{R})(\omega)$ is measurable in $\omega$.

Notice that for each $t\in \mathbb{R}$,
$$
\{\omega:N(\mathcal{R}\mid \mathcal{Q})(\omega)>t\}=\bigcup_{Q\in \mathcal{Q}}\{\omega:  N(Q, \mathcal{R})(\omega)>t\}.
$$
The result holds from the measurability of $N(Q,\mathcal{R})(\omega)$ in $\omega$.
\end{proof}

For any $\mathcal{R},\mathcal{Q},\mathcal{U},\mathcal{V}\in \mathfrak{P}(\mathcal{E}) $,
the following inequalities always hold.
\begin{align}
&N(\mathcal{R}\mid \mathcal{Q})(\omega)\leq N(\mathcal{U}\mid \mathcal{V})(\omega), \quad\quad  \text{if}\,\,\mathcal{U}\succ\mathcal{R},\, \mathcal{Q}\succ\mathcal{V},\label{n1}\\
&N(\Theta^{-1}\mathcal{R}\mid \Theta^{-1}\mathcal{Q})(\vartheta\omega)
\leq  N(\mathcal{R}\mid \mathcal{Q})(\omega),\label{n2}\\
&N( \mathcal{R}\vee\mathcal{Q}\mid  \mathcal{U})(\omega)
\leq  N(\mathcal{R}\mid \mathcal{U})(\omega)\cdot N(\mathcal{Q}\mid \mathcal{R}\vee \mathcal{U})(\omega),\label{n3}\\
&N( \mathcal{R}\vee\mathcal{Q}\mid \mathcal{U}\vee\mathcal{V})(\omega)\leq
N(\mathcal{R}\mid \mathcal{U})(\omega)\cdot N(\mathcal{Q}\mid \mathcal{V})(\omega).\label{n4}
\end{align}

Let $\mathcal{R}\in \mathfrak{U}(\mathcal{E})$ and  $ \mathcal{Q}\in \mathfrak{P}(\mathcal{E})$.
By the inequality \eqref{n2} and \eqref{n3}
it is easy to see that the sequence $\log N(\mathcal{R}^{(n)}\mid \mathcal{Q}^{(n)})(\omega)$ is subadditive for each $\omega$.
By the subadditive ergodic theorem (see \cite{Walters, Kifer}) the following limit
\begin{equation*}
h_{\Theta}(\mathcal{R}\mid \mathcal{Q})(\omega)=\lim_{n\rightarrow \infty}\frac{1}{n}\log N(\mathcal{R}^{(n)}\mid \mathcal{Q}^{(n)})(\omega)
\end{equation*}
$\mathbb{P}-$almost surely (a.s.) exists and
\begin{equation*}
h_{\Theta}(\mathcal{R}\mid \mathcal{Q})
=\lim_{n\rightarrow \infty}\frac{1}{n}\int \log N(\mathcal{R}^{(n)}\mid \mathcal{Q}^{(n)})(\omega)d\mathbb{P}
=\int h_{\Theta}(\mathcal{R}\mid \mathcal{Q})(\omega)d\mathbb{P}.
\end{equation*}
$h_{\Theta}(\mathcal{R}\mid \mathcal{Q})$ will be called {\it relative tail entropy of $\Theta$ on an open random cover $\mathcal{R}$ with respect to a random cover  $\mathcal{Q}$ }.
  If $\mathcal{Q}$ is a trivial random cover, then $h_{\Theta}(\mathcal{R}\mid \mathcal{Q})$ is called the {\it relative topological entropy $h^{(r)}_{\Theta}(\mathcal{R})$ of $\Theta$ with respect to an open random cover $\mathcal{R}$,} by \eqref{n1},
  \begin{equation}\label{n6}
   h^{(r)}_{\Theta}(\mathcal{R})\geq h_{\Theta}(\mathcal{R}\mid \mathcal{Q}),
   \end{equation}
  for all $\mathcal{Q}\in \mathfrak{P}(\mathcal{E})$.

From  \eqref{n1}, one can see that
\begin{equation}\label{n5}
h_{\Theta}(\mathcal{R}\mid \mathcal{Q})\leq h_{\Theta}(\mathcal{U}\mid \mathcal{V}) \quad \text{if}\,\,\mathcal{U}\succ\mathcal{R},\, \mathcal{Q}\succ\mathcal{V}.
\end{equation}
Then
there exists a limit (finite or infinite) over the directed set $\mathfrak{U}(\mathcal{E})$,
\begin{equation*}
h(\Theta\mid \mathcal{Q})=\lim_{\mathcal{R}\in \mathfrak{U}(\mathcal{E})}h_{\Theta}(\mathcal{R}\mid \mathcal{Q})=\sup_{\mathcal{R}\in \mathfrak{U}(\mathcal{E})}h_{\Theta}(\mathcal{R}\mid \mathcal{Q}),
\end{equation*}
which will be called the {\it relative tail entropy of $\Theta$ with respect to a random cover $\mathcal{Q}$. }
By the inequality \eqref{n5},
\begin{equation*}
h(\Theta\mid \mathcal{Q}) \leq h(\Theta\mid \mathcal{V}), \quad \text{if} \,\,\mathcal{Q}\succ \mathcal{V},
\end{equation*}
then one can take the limit again
\begin{equation*}
h^*(\Theta) =\lim_{\mathcal{Q}\in \mathfrak{P}(\mathcal{E})} h(\Theta\mid \mathcal{Q})
=\inf_{\mathcal{Q}\in \mathfrak{P}(\mathcal{E})}h(\Theta\mid \mathcal{Q}),
\end{equation*}
which is called the {\it relative tail entropy of $\Theta$}.
It follows from the inequality \eqref{n6} that
$h^{(r)}(\Theta)\geq h^*(\Theta)$.

\begin{remark}\label{remark1}
For each open cover  $\xi=\{A_1,\dots,A_k \}$ of the compact space $X$, $\{(\Omega\times A_i)\cap\mathcal{E}\}_{i=1}^k$  naturally form an open random cover of $\mathcal{E}$.
The relative tail entropy related with this kind of random cover is discussed under the name of ``relative conditional entropy" in \cite{Kifer2002}.
\end{remark}

One of our main goals is to establish the following  variational inequality, which shows that
the defect of upper semi-continuity of the relative measure-theoretical entropy function cannot exceed  the relative tail entropy.

\begin{theorem}\label{prop2}
Let  $S\times T$ be the continuous bundle RDS on  $\mathcal{H}$ and $m\in \mathcal{I}_{\mathbb{P}}(\mathcal{H})$.
Then $h^*_m(\Gamma\mid \mathcal{D}_{\mathcal{H}})\leq h^*(\Theta)$.
\end{theorem}

\begin{remark}
For the trivial space $(Y,\mathcal{C})$ and the random cover mentioned in Remark \ref{remark1}, the above result reduces to the theorem presented by Kifer and Liu (See \cite[Theorem 1.3.5]{KiferLiu2006})
\end{remark}

We will obtain the following variational principle when we consider the continuous bundle RDS $T\times T$.

\begin{theorem}\label{theo1}
Let $T$ be a continuous bundle RDS on $\mathcal{E}$. Then
$$
\max\{h^*_{\mu}(\Theta^{(2)} \mid \mathcal{A}_{\mathcal{E}^{(2)}}):\mu \in \mathcal{I}_{\mathbb{P}}(\mathcal{E}^{(2)})\}=h^*(\Theta).
$$

\end{theorem}

\begin{definition}
A continuous bundle RDS $T$ is called relatively asymptotically $h$-expansive if the relative tail entropy $h^*(\Theta)=0$.
\end{definition}

\begin{remark}
Theorem \ref{theo1} indicates that the upper semi-continuity of the function $ h_{(\cdot)}(\Theta^{(2)}\mid\mathcal{A}_{\mathcal{E}^{(2)}})$ on $\mathcal{I}_{\mathbb{P}}(\mathcal{E}^{(2)})$ is equivalent to relatively asymptotically $h$-expansiveness  of $T$.
Moreover, by Theorem \ref{prop2}, for a continuous bundle RDS $S\times T$ on $\mathcal{H}$ generated by the continuous bundle RDS $T$ and any other continuous bundle RDS $S$, relatively asymptotically $h$-expansiveness of $T$  is also equivalent to the upper semi-continuity of the function
$ h_{(\cdot)}(\Gamma\mid\mathcal{D}_{\mathcal{H}})$ on $\mathcal{I}_{\mathbb{P}}(\mathcal{H})$.
In general,  the upper semi-continuity of the usual measure-theoretical entropy does not imply the relatively asymptotically $h$-expansiveness of a random transformation, even in the deterministic case (see \cite[Example 6.4]{Misiurewicz}).
An equivalence condition with respect to the upper semi-continuity of the measure-theoretic entropy is given by making use of the local entropy theory (See \cite[Lemma 6.4]{HuangYi})

\end{remark}

As an application of the variational principle, we will derive the following result.

\begin{theorem}\label{th3}
Let $T, S$ be  two continuous bundle RDSs on $\mathcal{E}$ and $\mathcal{G}$, respectively. Suppose that $S$ is a principal extension of $T$ via the factor transformation $\pi$, then $h^*(\Lambda)=h^*(\Theta)$.
\end{theorem}

\begin{remark}
Theorem \ref{th3} shows that the relative tail entropy for random transformations could be  conserved by the principal extension. If two continuous bundle RDSs have a common principal extension,  they are equivalent in the sense of the principal extension.
\end{remark}

\section{Relative tail entropy and relative entropy}\label{sec3}

we will  first give two propositions regarded as the relative tail entropy, which will be needed in the proof of variational inequality later.

\begin{proposition}\label{power}
Let $T$ be a continuous bundle RDS on $\mathcal{E}$, and $\mathcal{Q}$ be a  random cover of $\mathcal{E}$. Then for each $m\in \mathbb{N}$,
\begin{equation*}
h(\Theta^m\mid \mathcal{Q}^{(m)})=mh(\Theta\mid \mathcal{Q}),
\end{equation*}
where $\mathcal{Q}^{(m)}= \bigvee_{i=0}^{m-1}(\Theta^i)^{-1}\mathcal{Q}$.
\end{proposition}

\begin{proof}
Let $\mathcal{R}$ be an open random  cover of $\mathcal{E}$. Since
\begin{equation*}
\bigvee_{j=0}^{n-1}(\Theta^{mj})^{-1}\big(\bigvee_{i=0}^{m-1}(\Theta^i)^{-1}\mathcal{R} \big)
=\bigvee_{i=0}^{nm-1}(\Theta^i)^{-1}\mathcal{R},
\end{equation*}
Then  for each $\omega\in \Omega$,
\begin{equation*}
N(\bigvee_{j=0}^{n-1}(\Theta^{mj})^{-1} \mathcal{R}^{(m)})\mid \bigvee_{j=0}^{n-1}(\Theta^{mj})^{-1} \mathcal{Q}^{(m)})(\omega)=N(\mathcal{R}^{(nm)}\mid \mathcal{Q}^{(nm)})(\omega).
\end{equation*}
By the definition of the  relative   tail entropy of $\Theta^m$ on open random  cover $\mathcal{R}^{(m)}$ with respect to $\mathcal{Q}^{(m)}$,
\begin{align*}
h_{\Theta^m}(\mathcal{R}^{(m)}\mid \mathcal{Q}^{(m)})
&=\lim_{n\rightarrow}\frac{1}{n}\int \log N(\bigvee_{j=0}^{n-1}(\Theta^{mj})^{-1} \mathcal{R}^{(m)})\mid \bigvee_{j=0}^{n-1}(\Theta^{mj})^{-1} \mathcal{Q}^{(m)})(\omega)d\mathbb{P}\\
&=\lim_{n\rightarrow}\frac{1}{n}\int \log
N(\mathcal{R}^{(nm)}\mid \mathcal{Q}^{(nm)})(\omega)
d\mathbb{P}\\
&=\lim_{n\rightarrow}m\frac{1}{nm}\int \log
N(\mathcal{R}^{(nm)}\mid \mathcal{Q}^{(nm)})(\omega)
d\mathbb{P}\\
&=mh_{\Theta}(\mathcal{R}\mid \mathcal{Q}).
\end{align*}
Then
\begin{equation*}
mh(\Theta \mid \mathcal{Q})= \sup_{\mathcal{R}}h_{\Theta^m}(\mathcal{R}^{(m)}\mid \mathcal{Q}^{(m)})\leq h(\Theta^m\mid \mathcal{Q}^{(m)}),
\end{equation*}
where the supremum is  taken over all open random covers $\mathcal{R}$  of $\mathcal{E}$.

Since $\mathcal{R}\prec \mathcal{R}^{(m)},$
then by the inequality \eqref{n1},
\begin{equation*}
N(\mathcal{R}^{(m)}\mid \bigvee_{j=0}^{n-1}(\Theta^{mj})^{-1} \mathcal{Q}^{(m)})(\omega)\leq N(\bigvee_{j=0}^{n-1}(\Theta^{mj})^{-1} \mathcal{R}^{(m)})\mid \bigvee_{j=0}^{n-1}(\Theta^{mj})^{-1} \mathcal{Q}^{(m)})(\omega),
\end{equation*}
which implies that
\begin{equation*}
h_{\Theta^m}(\mathcal{R}\mid  \mathcal{Q}^{(m)})\leq h_{\Theta^m}(\mathcal{R}^{(m)}\mid \mathcal{Q}^{(m)})= mh_{\Theta}(\mathcal{R}\mid \mathcal{Q}).
\end{equation*}
Thus
$ h(\Theta^m\mid  \mathcal{Q}^{(m)})\leq m h(\Theta\mid \mathcal{Q}) $ and the proposition is proved.
\end{proof}

We could deduce from Proposition \ref{power} the following power rule for the relative   tail entropy.

\begin{proposition}
Let $T$ be a continuous bundle RDS on $\mathcal{E}$. Then for each $m\in\mathbb{N}$,
\begin{equation*}
h^*(\Theta^m)=mh^*(\Theta).
\end{equation*}
\end{proposition}

\begin{proof}
By Proposition \ref{power},
\begin{equation*}
\inf_{\mathcal{Q}}h(\Theta^m\mid  \mathcal{Q}^{(m)})
=\inf_{\mathcal{Q}}m h(\Theta\mid \mathcal{Q})
= mh^*(\Theta),
\end{equation*}
where the infimum is taken over all random  covers  of $\mathcal{E}$.
Then $h^*(\Theta^m)\leq mh^*(\Theta).$

Since $\mathcal{Q}\prec \mathcal{Q}^{(m)}$, then
$$ h(\Theta^m\mid \mathcal{Q})\geq h(\Theta^m\mid \mathcal{Q}^{(m)})\geq mh^*(\Theta).$$
By taking infimum on the inequality over all  random  covers of  $\mathcal{E}$, one get
$h^*(\Theta^m)\geq mh^*(\Theta)$ and the equality holds.
\end{proof}

Let $\mu\in \mathcal{P}_{\mathbb{P}}(\mathcal{E})$. A partition $\mathcal{P}$ is called $\delta-$contains a partition $\mathcal{Q}$ if there exists a partition $\mathcal{R}\preceq \mathcal{P}$ such that $\inf\sum_i \mu(R^*_i\triangle Q^*_i)<\delta$, where the infimum is taken over all ordered partitions $\mathcal{R}^*, \mathcal{Q}^*$ obtained from $\mathcal{R}$ and $\mathcal{Q}$.

The following lemma essentially comes from the argument of Theorem 4.18 in \cite{Smorodinsky} and Lemma 4.15 in \cite{Walters}.

\begin{lemma}\label{lem415}
Given $\epsilon>0$ and $k\in\mathbb{N}$. There exists $\delta=\delta(\epsilon,k)>0$ such that if the measurable partition $\mathcal{P}$ $\delta-$contains $\mathcal{Q}$, where $\mathcal{Q}$ is a finite measurable partition with $k$ elements, then $H_{\mu}(\mathcal{Q}\mid \mathcal{P})<\epsilon.$
\end{lemma}

\begin{proof}
Let $\epsilon>0$. Choose $0<\delta<\frac{1}{e}$ such that $-\delta\log \delta+(1-\delta)\log(1-\delta)+\delta \log k<\epsilon$.
Suppose that $\mathcal{R}\preceq \mathcal{P}$ is the partition with $\sum_{i}\mu(R_i\triangle Q_i)<\delta$. One can construct a partition $\mathcal{S}$ by $S_0=\bigcup_i(R_i\cap Q_i)$ and $S_i=Q_i\setminus S_0$.
Since $\mathcal{R}\vee\mathcal{Q}=\mathcal{R}\vee\mathcal{S}$, and
\begin{equation*}
H_{\mu}(\mathcal{R})+H_{\mu}(\mathcal{Q}\mid \mathcal{R})=H_{\mu}(\mathcal{R}\vee \mathcal{Q})=H_{\mu}(\mathcal{R}\vee \mathcal{S})\leq H_{\mu}(\mathcal{S})+H_{\mu}(\mathcal{R}).
\end{equation*}
Then
\begin{equation*}
H_{\mu}(\mathcal{Q}\mid \mathcal{R})\leq H_{\mu}(\mathcal{S})\leq -\delta\log \delta+(1-\delta)\log(1-\delta)+\delta \log k<\epsilon,
\end{equation*}
and $H_{\mu}(\mathcal{Q}\mid \mathcal{P})<H_{\mu}(\mathcal{Q}\mid \mathcal{R})<\epsilon$.
\end{proof}

\begin{remark}
We discuss  here the conditional entropy  instead of the usual measure-theoretic  entropy  in \cite{Smorodinsky}.
 The  result does not require that the two partitions have the same cardinality, which is a little different from Lemma 4.15 in \cite{Walters}.
 \end{remark}

\begin{lemma}\label{Fact666}
Let $\mu^{(i)} \in  \mathcal{P}_{\mathbb{P}}(\mathcal{E}), i\in \mathbb{N}$ and $\delta=\delta(\omega)$ be a positive random variable on $\Omega$. There exists a finite measurable partition $\mathcal{R}=\{R \}$ of $\mathcal{E}$ such that $\text{diam}\, R (\omega) \leq \delta(\omega)$ $\mathbb{P}-$a.s. and $\mu^{(i)}(\partial R  )=0$  for each $i\in \mathbb{N}, R\in \mathcal{R} $ in the sense of $\mu^{(i)}(\partial  R )=\int\mu_{\omega}^{(i)}(\partial R (\omega))d {\mathbb{P}(\omega)}$, where $\partial$ denotes the boundary.

\end{lemma}

\begin{proof}

Since $(\Omega,\mathcal{F},\mathbb{P})$ is a Lebesgue space, it can be viewed as a Borel subset of the unit interval $[0,1]$, and $\mu^{(i)}, i\in \mathbb{N}$ are also probability measures on the compact space $[0,1]\times X$ with the marginal $\mathbb{P}$ on $[0,1]$.

Fix a point $(t,x)\in [0,1]\times X$ and $\mu^{(i)}\in \mathcal{P}_{\mathbb{P}}(\mathcal{E})$. For each nonrandom $\epsilon >0$,
$\partial B((t,x),\epsilon)\subset \overline{B}((t,x),\epsilon)\setminus \text{int} B((t,x),\epsilon)$, where $B((t,x),\epsilon)$ is the open ball of center at $(t,x)$ and radius $\epsilon$ with the product metric $d=(d_1^2+d_2^2)^{\frac{1}{2}}$ on $[0,1]\times X$ and $\overline{B}$ denotes the closure of $B$. If $\epsilon_1\neq \epsilon_2$,
$\big(\overline{B}((t,x),\epsilon_1)\setminus \text{int} B((t,x),\epsilon_1)\big)\cap \big(\overline{B}((t,x),\epsilon_2)\setminus \text{int} B((t,x),\epsilon_2)\big)=\emptyset$. Then  there exists only at most countably many of $\epsilon_j $ such that $\mu^{(i)}(\overline{B}((t,x),\epsilon_j)\setminus \text{int} B((t,x),\epsilon_j))>0, j=1, 2, \dots$. It follows that for each $(t,x)\in [0,1]\times X$ one could choose some positive real number $\gamma(t)<\frac{\delta(t)}{2}$ such that for all $i\in \mathbb{N}$,
$\mu^{(i)}(\overline{B}((t,x),\gamma(t))\setminus \text{int} B((t,x),\gamma(t)))=0$ and then $\mu^{(i)}(\partial B((t,x),\gamma(t)))=0$.

By the compactness of $[0,1]\times X$, there exists an open cover of $[0,1]\times X$ by finite many open ball $B_1, \dots, B_k$ with the diameter of the $t-$section of $B_j$ $\text{diam}\, B_j(t)<\delta(t)$ and
$\mu^{(i)}(\partial B_j)=\int \mu^{(i)}_{\omega}\partial B_j(\omega)d\mathbb{P}=0$ for each $1\leq j\leq k$ and $i\in \mathbb{N}$. Then $B_1, B_2\setminus B_1, \dots, B_k\setminus \cup_{j=1}^{k-1}B_j$ forms a measurable partition of $[0,1]\times X$. Let $A_1=B_1$, and $A_n=B_n\setminus \cup_{j=1}^{n-1}B_j$ for each $1\leq n\leq k$ and denote by $\xi=\{A_1,\dots, A_k\}$. Then $\xi$ is a measurable partition of $[0,1]\times X$ and for each $1\leq j\leq k$, $\text{diam}\, A_j(t)<\delta (t)$ for each $t\in [0,1]$. Since $\partial A_n\subset \cup_{j=1}^n\partial B_j$, then $\mu^{(i)}(\partial A_j)=0$, for each $1\leq j\leq k$ and $i\in \mathbb{N}$.

Let $R_j=A_j\cap \mathcal{E}$ and $R_j(\omega)=\{x:(\omega,x)\in R_j\}$. Notice that the marginal $\mathbb{P}$ is supported on $\Omega$, then $\mu^{(i)}(\partial R_j)=0$ and  $\mathcal{R}=\{R_1,\dots,R_k\}$ is the measurable partition as desired.
\end{proof}

Let $T$ and $S$ be two continuous bundle RDSs on $\mathcal{E}$ and $\mathcal{G}$, respectively.
Let $T$ be a factor of $S$ via the factor transformation $\pi$.
The transformation $\pi$  induce a map, which is again denoted by $\pi$, from  $\mathcal{P}_{\mathbb{P}}(\mathcal{G})$ to $\mathcal{P}_{\mathbb{P}}(\mathcal{E})$ by $\pi\mu=\mu\pi^{-1}$. This induced map $\pi$ transports every measure on $ \mathcal{G} $ to a measure $\pi\mu $  in $\mathcal{P}_{\mathbb{P}}(\mathcal{E})$. The following proposition is classical in the deterministic dynamical system \cite{Denker1976}.

\begin{proposition}\label{propm1}
The induced map $\pi$ is a continuous affine map from $\mathcal{P}_{\mathbb{P}}(\mathcal{G}) $ onto $\mathcal{P}_{\mathbb{P}}(\mathcal{E})$.
\end{proposition}
\begin{proof}
If $\mu_n \rightarrow \mu\in \mathcal{P}_{\mathbb{P}}(\mathcal{G})$,
then $\int f  d\mu_n \rightarrow \int f d\mu$ for all $f \in \mathcal{C}(\mathcal{G})$, where $\mathcal{C}(\mathcal{G})$ is the set of random continuous functions on $\mathcal{G}$ (see \cite{Crauel}), and therefore $\int g\circ\pi d\mu_n\rightarrow \int g\circ\pi d\mu$ for all $g\in \mathcal{C}(\mathcal{G})$ by the measurability of the factor transformation $\pi$. This implies $\pi\mu_n\rightarrow \pi\mu$.

It is clear that $\pi(\alpha\mu+(1-\alpha)\nu)=\alpha\pi\mu+(1-\alpha)\pi\nu$ for all $\mu,\nu\in \mathcal{P}_{\mathbb{P}}(\mathcal{G})$ and $0\leq \alpha\leq 1$.

Since $\mathcal{P}_{\mathbb{P}}(\mathcal{G})$ is a compact and convex subset of $\mathcal{P}_{\mathbb{P}}(\Omega\times Y)$ (see \cite[Section 1.5]{Arnold}), and the ergodic measures on $\mathcal{G}$ are just the point measures if we take the identity transformation on $\mathcal{G}$, {\it i.e.}, $Id(\omega,x)=(\omega,x)$. Then by the Krein-Millman theorem (see \cite[P440]{Dunford}), the convex combinations of point measures are dense in $\mathcal{P}_{\mathbb{P}}(\mathcal{G})$. it follows that $\pi:\mathcal{G}\rightarrow \mathcal{E} $ is onto that $\pi:\mathcal{P}_{\mathbb{P}}(\mathcal{G})\rightarrow \mathcal{P}_{\mathbb{P}}(\mathcal{E})$ is onto.
\end{proof}

It is not hard to verify that the induced map $\pi$ send a $\Lambda$-invariant measure in $\mathcal{I}_{\mathbb{P}}(\mathcal{G})$ to a $\Theta$-invariant measure in $\mathcal{I}_{\mathbb{P}}(\mathcal{E})$. The following result shows that this map is also surjective.

\begin{proposition}\label{prop33}
Let $\mu\in \mathcal{I}_{\mathbb{P}}(\mathcal{E})$. There exists $m\in \mathcal{I}_{\mathbb{P}}(\mathcal{G})$ with $\pi m=\mu$
\end{proposition}
\begin{proof}
By Proposition \ref{propm1}, there exists a $\nu\in \mathcal{P}_{\mathbb{P}}(\mathcal{G})$ such that $\pi\nu=\mu$. Since $\Theta\mu=\mu$ and $\pi\Lambda=\Theta \pi$, one has $\pi(\Lambda\nu)=\mu$, and more generally, $\pi(\Lambda^n\nu)=mu$. By the affinity of $\pi$, $\pi(\frac{1}{n}\sum_{i=0}^{n-1}\Lambda^i\nu)=\mu$. Denote by  $\nu^{(n)}=\frac{1}{n}\sum_{i=0}^{n-1}\Lambda^i\nu$ and let $m$ be one limit point of the sequence $\nu^{(n)}$, It follows from Theorem 1.5.8 in \cite{Arnold} that $m\in \mathcal{I}_{\mathbb{P}}(\mathcal{G})$. Since $\pi$ is continuous, then $\pi m=\mu$.
\end{proof}

We need the following lemma (see \cite[Section 14.3]{Glasner}) which follows from the martingale convergence theorem.
\begin{lemma}\label{lem222}
Let $\mu\in \mathcal{P}_{\mathbb{P}}(\mathcal{G})$, $\mathcal{R}=\{R_1,\dots,R_k\}$ be a finite  measurable partition of $\mathcal{G}$ with $H_{\mu}(\mathcal{R})<\infty$ and $\mathcal{A}_1\prec\cdots\prec\mathcal{A}_n\prec\cdots$ be an increasing sequence of sub-$\sigma$-algebra of $\mathcal{A}$ with $\bigvee_{n=1}^{\infty}=\mathcal{A}$. Then
\begin{equation*}
H_{\mu}(\mathcal{R}\mid\mathcal{A}_{\mathcal{G}})
=\lim_{n\rightarrow\infty}H_{\mu}(\mathcal{R}\mid\mathcal{A}_n)
=\inf_n H_{\mu}(\mathcal{R}\mid\mathcal{A}_n).
\end{equation*}
\end{lemma}

The following result is a relative version of Lemma 6.6.7 in \cite{Down2011}. Similar results for random transformations could be found in \cite{LedWal,Kifer2001}

\begin{lemma}\label{lem22}
Let $(\mathcal{E},\Theta)$ be a factor of $(\mathcal{G},\Lambda)$ via a factor transformation $\pi$, $m\in \mathcal{P}_{\mathbb{P}}(\mathcal{G})$ and $\mathcal{R}=\{R\}$ be a finite measurable partition of $\mathcal{G}$ with $m(\partial R)=0$, where $\partial$ denotes the boundary and $m(\partial R)=\int m_{\omega}(\partial R(\omega))d\mathbb{P}$.
Then
\begin{enumerate}
\item[(i)]
$m$ is a supper semi-continuity point of the function $\mu\rightarrow H_{\mu}(\mathcal{R}\mid \mathcal{A}_{\mathcal{G}})$ defined on $\mathcal{P}_{\mathbb{P}}(\mathcal{G})$, {\it i.e.},
$$
\limsup_{\mu\rightarrow m}H_{\mu}(\mathcal{R}\mid \mathcal{A}_{\mathcal{G}})\leq H_m(\mathcal{R}\mid \mathcal{A}_{\mathcal{G}}).
$$

\item[(ii)]
If $m\in \mathcal{I}_{\mathbb{P}}(\mathcal{G})$, the function $\mu\rightarrow h_{\mu}(\mathcal{R}\mid \mathcal{A}_{\mathcal{G}})$ defined on $\mathcal{I}_{\mathbb{P}}(\mathcal{G})$ is upper semi-continuous at $m$, {\it i.e.},
$$
\limsup_{\mu\rightarrow m}h_{\mu}(\mathcal{R}\mid \mathcal{A}_{\mathcal{G}})\leq h_m(\mathcal{R}\mid \mathcal{A}_{\mathcal{G}}).
$$

\end{enumerate}
\end{lemma}

\begin{proof}
(i) For $R\in\mathcal{R}$ with $R(\omega)=\{x: (\omega,x)\in R\}$. Let $\overline{R}=\{(\omega,x):x\in \overline{R(\omega)}\}$ and $\underline{R}=\{(\omega,x):x\in  \text{int}( R(\omega))\}$, where $\overline{R(\omega)}$  and $\text{int}( R(\omega))$ denotes the closure and the interior of $R(\omega)$, respectively.
Then $\overline{R}$ is a closed  random  set of $\mathcal{G}$ and $\underline{R}$ is an open random set. By Portmenteau theorem (see \cite{Crauel}),
\begin{equation*}
m(\overline{R})\geq \limsup_{\mu\rightarrow m}\mu(\overline{R})\geq \limsup_{\mu\rightarrow m}\mu(R)
\geq \liminf_{\mu\rightarrow m}\mu(R)\geq \liminf_{\mu\rightarrow m}\mu(\underline{R})\geq m(\underline{R}).
\end{equation*}
Since $m(\overline{R})=m(R)=m(\underline{R})$ by $m(\partial R)=0$, then
$\mu\rightarrow \mu(R)$ defined on $\mathcal{P}_{\mathbb{P}}(\mathcal{G})$ is continuous at $m$.
Recall that the function $t\rightarrow -t\log t$ is continuous on $[0,1]$. Then
$\mu\rightarrow H_{\mu}(\mathcal{R})$ is also continuous at $m$ on $\mathcal{P}_{\mathbb{P}}(\mathcal{G})$.
Moreover, if $\mathcal{Q}=\{Q\}$ is a measurable partition of $\mathcal{G}$ with $m(\partial Q)=0$ for each $Q\in \mathcal{Q}$, then the conditional entropy $\mu\rightarrow H_{\mu}(\mathcal{R}\mid \mathcal{Q})$ of the partition $\mathcal{R}$ over $\mathcal{Q}$ is continuous at $m$.

Let $\nu=\pi m$. By Lemma \ref{Fact666}, there exists a refining sequence of finite measurable partitions $\mathcal{Q}_k=\{Q_{k_i}\}$ of $\mathcal{E}$ satisfying $\nu(\partial Q_{k_i})=0$, for each $Q_{k_i}\in \mathcal{Q}_k$, $k=1,2,\dots$. Then $\{\pi^{-1}\mathcal{Q}_k\}$  is a refining sequence of measurable partitions of $\mathcal{G}$, all having the boundary of measure zero at $m$.
It follows that for each $k\in \mathbb{N}$, the function $\mu\rightarrow H_{\mu}(\mathcal{R}\mid \pi^{-1}\mathcal{Q}_k)$ is continuous at $m$. Notice that $\bigvee_{k=1}^{\infty}\mathcal{Q}_k=\mathcal{A}$ and $H_{\mu}(\mathcal{R}\mid \pi^{-1}\mathcal{Q}_k)$ decrease in $k$.
Thus the function $\mu\rightarrow \inf_k H_{\mu}(\mathcal{R}\mid \pi^{-1}\mathcal{Q}_k)$ is upper semi-continuous at $m$ and the property (i) follows from Lemma \ref{lem222}.

(ii) Let $n\in \mathbb{N}$. Since the function $\mu\rightarrow \frac{1}{n}H_{\mu}(\mathcal{R}^{(n)}\mid \pi^{-1}\mathcal{Q}_k)$ is also continuous at $m$ for each $k=1,2,
\dots,$ where $\mathcal{R}^{(n)}=\bigvee_{i=0}^{n-1}(\Lambda^i)^{-1}\mathcal{R}$, then the function
$\mu\rightarrow \inf_k\frac{1}{n}H_{\mu}(\mathcal{R}^{(n)}\mid \pi^{-1}\mathcal{Q}_k)=\frac{1}{n}H_{\mu}(\mathcal{R}^{(n)}\mid \mathcal{A}_{\mathcal{G}})$ is upper semi-continuous at $m$. Therefore the function $\mu\rightarrow \inf_n\frac{1}{n}H_{\mu}(\mathcal{R}^{(n)}\mid \mathcal{A}_{\mathcal{G}})=h_{\mu}(\mathcal{R}\mid \mathcal{A}_{\mathcal{G}})$ is upper semi-continuous at $m$ and the property (ii) holds.
\end{proof}

We need the following lemma which shows the basic connection between the relative entropy and relative tail entropy.

\begin{lemma}\label{lemlog}
Let  $S$ be  a continuous bundle RDS on  $\mathcal{G}$. Suppose that $\mathcal{R}=\{R\} , \mathcal{Q}=\{Q\}$ are two finite measurable partitions of $\mathcal{G}$ and $\mu\in \mathcal{P}_{\mathbb{P}}(\mathcal{G})$, then
\begin{equation*}
H_{\mu}(\mathcal{R}\mid \mathfrak{Q}\vee \mathcal{F}_{\mathcal{G}})\leq
\int \log N(\mathcal{R}\mid \mathcal{Q})(\omega)d\mathbb{P},
\end{equation*}
where $\mathfrak{Q}$ is the sub-$\sigma$-algebra generated by the partition $\mathcal{Q}$ and $\mathcal{F}_{\mathcal{G}}=\{(F\times Y)\cap \mathcal{G}:F\in \mathcal{F}   \}$.

\end{lemma}

\begin{proof}
A simple calculation (see \cite[Section 14.2]{Glasner}) shows that
\begin{equation*}
E(1_R\mid \mathfrak{Q}\vee \mathcal{F}_{\mathcal{G}})
=\sum_{Q\in \mathcal{Q}}1_Q\frac{E(1_{R\cap Q}\mid \mathcal{F}_{\mathcal{G}} )}{E(1_Q\mid \mathcal{F}_{\mathcal{G}})}.
\end{equation*}
Then
\begin{align*}
H_{\mu}(\mathcal{R}\mid \mathfrak{Q}\vee \mathcal{F}_{\mathcal{G}})
&=\int \sum_{R\in \mathcal{R}}-1_R\log E(1_R\mid \mathfrak{Q}\vee \mathcal{F}_{\mathcal{G}})d\mu\\
&=\int \sum_{R\in \mathcal{R}}-1_R\log \sum_{Q\in \mathcal{Q}}1_Q\frac{E(1_{R\cap Q}\mid \mathcal{F}_{\mathcal{G}} )}{E(1_Q\mid \mathcal{F}_{\mathcal{G}})}d\mu\\
&=\int \sum_{R\in \mathcal{R}}-1_R \sum_{Q\in \mathcal{Q}}1_Q\log\frac{E(1_{R\cap Q}\mid \mathcal{F}_{\mathcal{G}} )}{E(1_Q\mid \mathcal{F}_{\mathcal{G}})}d\mu\\
&=\int \sum_{R\in \mathcal{R}}\sum_{Q\in \mathcal{Q}}-1_{R\cap Q}\log\frac{E(1_{R\cap Q}\mid \mathcal{F}_{\mathcal{G}} )}{E(1_Q\mid \mathcal{F}_{\mathcal{G}})}d\mu\\
&=\int\sum_{R\in \mathcal{R}}\sum_{Q\in \mathcal{Q}}- E(1_{R\cap Q}
\log\frac{E(1_{R\cap Q}\mid \mathcal{F}_{\mathcal{G}} )}{E(1_Q\mid \mathcal{F}_{\mathcal{G}})} \mid \mathcal{F}_{\mathcal{G}})d\mu\\
&=\int\sum_{R\in \mathcal{R}}\sum_{Q\in \mathcal{Q}}- E(1_{R\cap Q}\mid\mathcal{F}_{\mathcal{G}})
\log\frac{E(1_{R\cap Q}\mid \mathcal{F}_{\mathcal{G}} )}{E(1_Q\mid \mathcal{F}_{\mathcal{G}})}d\mu
\end{align*}
Since $\mu$ could disintegrate $d\mu(\omega,y)=d\mu_{\omega}(y)d\mathbb{P}(\omega)$,
$E(1_{R\cap Q}\mid\mathcal{F}_{\mathcal{G}})=\mu_{\omega}((R\cap Q)(\omega))$ and $E(1_Q\mid \mathcal{F}_{\mathcal{G}})=\mu_{\omega}(Q(\omega))$ $\mathbb{P}-$a.s., then
\begin{align*}
 H_{\mu}(\mathcal{R}\mid \mathfrak{Q}\vee \mathcal{F}_{\mathcal{G}})
 &=\int\sum_{R\in \mathcal{R}}\sum_{Q\in \mathcal{Q}}- \mu_{\omega}((R\cap Q)(\omega))\log\frac{\mu_{\omega}((R\cap Q)(\omega))}{\mu_{\omega}(Q(\omega))}d\mathbb{P}\\
 &=\int\sum_{Q\in \mathcal{Q}}\mu_{\omega}(Q(\omega))
 \big(
 -\sum_{R\in \mathcal{R}} \frac{\mu_{\omega}((R\cap Q)(\omega))}{\mu_{\omega}(Q(\omega))}\log\frac{\mu_{\omega}((R\cap Q)(\omega))}{\mu_{\omega}(Q(\omega))}\big)d\mathbb{P}\\
\end{align*}
Notice that
\begin{equation*}
-\sum_{R\in \mathcal{R}} \frac{\mu_{\omega}((R\cap Q)(\omega))}{\mu_{\omega}(Q(\omega))}\log\frac{\mu_{\omega}((R\cap Q)(\omega))}{\mu_{\omega}(Q(\omega))}\big)\leq \log N(Q, \mathcal{R})(\omega).
\end{equation*}
Thus
\begin{align*}
H_{\mu}(\mathcal{R}\mid \mathfrak{Q}\vee \mathcal{F}_{\mathcal{G}})
\leq \int \sum_{Q\in \mathcal{Q}}\mu_{\omega}(Q(\omega))\log N(Q, \mathcal{R})(\omega)d\mathbb{P}\leq \int \log N(\mathcal{R}\mid\mathcal{Q})(\omega)d\mathbb{P}.
\end{align*}
\end{proof}

\begin{remark}
When we consider the relative entropy $H_{\mu}(\mathcal{R}\mid \mathfrak{Q})$ with respect to two measurable partitions $\mathcal{R}$ and $\mathcal{Q}$, it is not hard to see that $H_{\mu}(\mathcal{R}\mid \mathfrak{Q})\leq N(\mathcal{R}\mid\mathcal{Q})$, which is similar to the case in the deterministic system. Moreover, the iteration of the random transformation is not necessary in this lemma, though we assume that the condition is in the environment of random dynamical systems.
\end{remark}

\section{Variational principle for relative tail entropy }\label{sec4}

We now take up the consideration of the relationship between the relative entropy and relative tail entropy on the measurable subset $\mathcal{H}$ of $\Omega\times Y\times X$ with respect to the product $\sigma-$algebra $\mathcal{F}\times\mathcal{C}\times \mathcal{B}$.
The following result follows from Lemma \ref{lem22} directly.

\begin{lemma}\label{lem2}
Let $\mathcal{R}=\{R_1,\dots, R_k\}$ be a finite measurable partition of $\mathcal{H}$. Given $m\in \mathcal{P}_{\mathbb{P}}(\mathcal{H})$ satisfying $m(\partial R_i)=0$ for each $1\leq i\leq k$, then
$m$ is a upper semi-continuity point of the function $\mu\rightarrow H_{\mu}(\mathcal{R}\mid  \mathcal{D}_{\mathcal{H}})$ defined on $ \mathcal{P}_{\mathbb{P}}(\mathcal{H})$, {i.e.},
\begin{equation*}
\limsup_{\mu\rightarrow m}H_{\mu}(\mathcal{R}\mid  \mathcal{D}_{\mathcal{H}})\leq H_m(\mathcal{R}\mid  \mathcal{D}_{\mathcal{H}}).
\end{equation*}
\end{lemma}

\begin{lemma}\label{lem3}
Let $S\times T$ be the continuous bundle RDSs on  $\mathcal{H}$ and $\mu\in\mathcal{P}_{\mathbb{P}}(\mathcal{H})$.
 Suppose that $\mathcal{R}, \mathcal{Q}$ are two finite measurable partitions of $\mathcal{H}$, Then
 \begin{equation*}
H_{\mu}(\mathcal{R}\mid \mathcal{D}_{\mathcal{H}})
\leq H_{\mu}(\mathcal{Q}\mid \mathcal{D}_{\mathcal{H}}) +
\int\log N(\mathcal{R}\mid\mathcal{Q})(\omega)d\mathbb{P}.
\end{equation*}

\end{lemma}

\begin{proof}
Since $\mathcal{F}_{\mathcal{H}}$ is a sub-$\sigma$-algebra of $\mathcal{D}_{\mathcal{H}}$, then
$\mathcal{D}_{\mathcal{H}}\vee\mathcal{F}_{\mathcal{H}}=\mathcal{D}_{\mathcal{H}}$.
Let $\mathfrak{Q} $ be the sub-$\sigma$-algebra generated by the partition $\mathcal{Q}$.
By Lemma \ref{lemlog},
\begin{align*}
H_{\mu}(\mathcal{R}\mid \mathcal{D}_{\mathcal{H}})
&=H_{\mu}(\mathcal{R}\mid \mathcal{D}_{\mathcal{H}}\vee\mathcal{F}_{\mathcal{H}})\\
&\leq H_{\mu}(\mathcal{Q}\mid \mathcal{D}_{\mathcal{H}}\vee\mathcal{F}_{\mathcal{H}})
+H_{\mu}(\mathcal{R}\mid \mathfrak{Q} \vee \mathcal{D}_{\mathcal{H}}\vee\mathcal{F}_{\mathcal{H}})
\\
&\leq H_{\mu}(\mathcal{Q}\mid \mathcal{D}_{\mathcal{H}}\vee\mathcal{F}_{\mathcal{H}})
+H_{\mu}(\mathcal{R}\mid \mathfrak{Q}  \vee\mathcal{F}_{\mathcal{H}})
\\
&\leq H_{\mu}(\mathcal{Q}\mid \mathcal{D}_{\mathcal{H}})+
\int \log N(\mathcal{R}\mid\mathcal{Q})(\omega)d\mathbb{P},
\end{align*}
 and the result holds.
\end{proof}

\begin{proposition}\label{prop1}
  Let $S\times T$ be the continuous bundle RDS on  $\mathcal{H}$ and $\mu\in\mathcal{I}_{\mathbb{P}}(\mathcal{H})$. Then for  each finite measurable partition $\mathcal{Q}$ of $\mathcal{E}$,
  $$h_{\mu}(\Gamma\mid\mathcal{D}_{\mathcal{H}} )\leq h_{\mu}(\pi_{\mathcal{E}}^{-1}\mathcal{Q}\mid \mathcal{D}_{\mathcal{H}})+h(\Theta\mid\mathcal{Q}).$$
\end{proposition}

\begin{proof}
Let $\mathcal{R}=\{R_1,\dots,R_k\}$ be a measurable partition of $\mathcal{E}$ and $\nu=\pi_{\mathcal{E}}\mu$.

Recall that $(\Omega,\mathcal{F},\mathbb{P})$ can be view as a Borel subset of the unit interval $[0,1]$. Then $\nu\in \mathcal{P}_{\mathbb{P}}(\mathcal{E})$ is also a probability measure on the compact space $[0,1]\times X$ with the marginal $\mathbb{P}$ on $[0,1]$. Let $\epsilon >0$ and $\delta>0$ as desired in Lemma \ref{lem415}. Since $\nu$ is regular, there exists a compact subset $P_i\subset R_i$ with $\nu(R_i\setminus P_i)<\frac{\delta}{2k}$ for each $1\leq i\leq k$.
Denote by
$P_0=\mathcal{E}\setminus \bigcup_{i=1}^kP_i$. Then $\mathcal{P}=\{P_0,P_1,\dots, P_k\}$ is a measurable partition of $\mathcal{E}$ and
$\sum_{i=1}^k\nu(R_i\setminus P_i)+\nu (P_0)< \frac{\delta}{2k}\cdot k+\frac{\delta}{2}=\delta.$
By Lemma \ref{lem415}, $H_{\nu}(\mathcal{R}\mid \mathcal{P})<\epsilon.$
Let $U_i=P_0\cup P_i$. Since for each $\omega\in \Omega$, $(P_0\cup P_i)\omega$ is an open subset of $\mathcal{E}_{\omega}$, then $\mathcal{U}=\{U_1,\dots, U_k\}$ is an open random  cover of $\mathcal{E}$,
and $N(\mathcal{P}\mid \mathcal{U})(\omega)\leq N(\mathcal{P}\mid \mathcal{U})\leq 2.$
Then
\begin{align*}
h_{\mu}(\pi_{\mathcal{E}}^{-1}\mathcal{R}\mid \mathcal{D}_{\mathcal{H}})
&\leq h_{\mu}(\pi_{\mathcal{E}}^{-1}\mathcal{P}\mid \mathcal{D}_{\mathcal{H}})
+H_{\mu}(\pi_{\mathcal{E}}^{-1}\mathcal{R}\mid \pi_{\mathcal{E}}^{-1}\mathcal{P})\\
&= h_{\mu}(\pi_{\mathcal{E}}^{-1}\mathcal{P}\mid \mathcal{D}_{\mathcal{H}})
+ H_{\nu}(\mathcal{R}\mid \mathcal{P})\\
&\leq h_{\mu}(\pi_{\mathcal{E}}^{-1}\mathcal{P}\mid \mathcal{D}_{\mathcal{H}})
+ \epsilon.
\end{align*}
By Lemma \ref{lem3}, one has
\begin{equation*}
H_{\mu}(\pi_{\mathcal{E}}^{-1}\mathcal{P}\mid \mathcal{D}_{\mathcal{H}})
\leq H_{\mu}(\pi_{\mathcal{E}}^{-1}\mathcal{Q}\mid \mathcal{D}_{\mathcal{H}})
+\int \log N(\mathcal{P}\mid \mathcal{Q})(\omega)d\mathbb{P}.
\end{equation*}
Notice that for each $\omega\in \Omega$,
$N(\mathcal{P}\mid \mathcal{Q})(\omega)\leq N(\mathcal{U}\mid \mathcal{Q})(\omega)\cdot N(\mathcal{P}\mid \mathcal{U})(\omega) $.
Then
\begin{align*}
H_{\mu}(\pi_{\mathcal{E}}^{-1}\mathcal{P}\mid \mathcal{D}_{\mathcal{H}})
\leq H_{\mu}(\pi_{\mathcal{E}}^{-1}\mathcal{Q}\mid \mathcal{D}_{\mathcal{H}})
+\int \log &N(\mathcal{U}\mid \mathcal{Q})(\omega)d\mathbb{P}\\
&+\int\log N(\mathcal{P} \mid \mathcal{U})(\omega)d\mathbb{P}.
\end{align*}
Applying the above result to $\bigvee_{i=0}^{n-1}(\Theta^i)^{-1}\mathcal{P},  \bigvee_{i=0}^{n-1}(\Theta^i)^{-1}\mathcal{Q},$ and $\bigvee_{i=0}^{n-1}(\Theta^i)^{-1}\mathcal{U}, $
dividing by $n$ and letting $n\rightarrow \infty$, one obtain
\begin{align*}
h_{\mu}(\pi_{\mathcal{E}}^{-1}\mathcal{P}\mid \mathcal{D}_{\mathcal{H}})
\leq h_{\mu}(&\pi_{\mathcal{E}}^{-1}\mathcal{Q}\mid \mathcal{D}_{\mathcal{H}})
+h_{\Theta}(\mathcal{U}\mid \mathcal{Q})\\
&+ \lim_{n\rightarrow\infty}\frac{1}{n}\int \log N\big(\bigvee_{i=0}^{n-1}(\Theta^i)^{-1}\mathcal{P}\mid \bigvee_{i=0}^{n-1}(\Theta^i)^{-1}\mathcal{U}\big)d\mathbb{P}.
\end{align*}
Observe that
\begin{equation*}
\lim_{n\rightarrow\infty}\frac{1}{n}\int \log N\big(\bigvee_{i=0}^{n-1}(\Theta^i)^{-1}\mathcal{P}\mid \bigvee_{i=0}^{n-1}(\Theta^i)^{-1}\mathcal{U}\big)d\mathbb{P}\leq \log 2,
\end{equation*}
and $h_{\Theta}(\mathcal{U}\mid \mathcal{Q})\leq h(\Theta\mid \mathcal{Q})$,
then
\begin{align*}
h_{\mu}(\pi_{\mathcal{E}}^{-1}\mathcal{R}\mid \mathcal{D}_{\mathcal{H}})
&\leq h_{\mu}(\pi_{\mathcal{E}}^{-1}\mathcal{P}\mid \mathcal{D}_{\mathcal{H}})+H_{\mu}(\pi_{\mathcal{E}}^{-1}\mathcal{R}\mid \pi_{\mathcal{E}}^{-1}\mathcal{P})\\
&\leq h_{\mu}(\pi_{\mathcal{E}}^{-1}\mathcal{Q}\mid \mathcal{D}_{\mathcal{H}})
+h(\Theta\mid \mathcal{Q})+\log 2+\epsilon.
\end{align*}

Let $\mathcal{R}_1\prec\cdots\prec\mathcal{R}_n\prec \cdots$ be an increasing sequence of finite measurable partitions with $\bigvee_{i=1}^{\infty}\mathcal{R}_n=\mathcal{A}$, by Lemma 1.6 in \cite{Kifer}, one has
\begin{equation}\label{ineq2}
h_{\mu}(\Gamma\mid\mathcal{D}_{\mathcal{H}})\leq  h_{\mu}(\pi_{\mathcal{E}}^{-1}\mathcal{Q}\mid \mathcal{D}_{\mathcal{H}})
+h(\Theta\mid \mathcal{Q})+\log 2+\epsilon.
\end{equation}

Since
\begin{equation*}
H_{\mu}\big(
\bigvee_{j=0}^{n-1}(\Gamma^{mj})^{-1}(\bigvee_{i=0}^{m-1}(\Gamma^i)^{-1}\pi^{-1}_{\mathcal{E}}\mathcal{Q})
\mid \mathcal{D}_{\mathcal{H}}
\big)
=H_{\mu}\big(
\bigvee_{i=0}^{nm-1}(\Gamma^i)^{-1}\pi^{-1}_{\mathcal{E}}\mathcal{Q}\mid \mathcal{D}_{\mathcal{H}}
\big).
\end{equation*}
It is not hard to see that
\begin{equation}\label{eq2}
h_{\mu,\Gamma^{^m}}(\bigvee_{i=0}^{m-1}(\Gamma^i)^{-1}\pi^{-1}_{\mathcal{E}}\mathcal{Q}
\mid \mathcal{D}_{\mathcal{H}})
=mh_{\mu}(\pi_{\mathcal{E}}^{-1}\mathcal{Q}\mid \mathcal{D}_{\mathcal{H}} ),
\end{equation}
where $h_{\mu,\Gamma^{^m}}(\xi\mid \mathcal{D}_{\mathcal{H}})$ denotes the relative entropy of $\Gamma^m $ with respect to the partition $\xi$.

By Lemma 1.4 in \cite{Kifer}, for each $m\in \mathbb{N}$,
\begin{equation}\label{eq3}
h_{\mu}(\Gamma^m\mid \mathcal{D}_{\mathcal{H}})=mh_{\mu}(\Gamma\mid \mathcal{D}_{\mathcal{H}}),
\end{equation}
where $h_{\mu}(\Gamma^m\mid \mathcal{D}_{\mathcal{H}})$ is the relative entropy of $\Gamma^m$.

By the equality \eqref{eq2}, \eqref{eq3} and Proposition \ref{power}, and applying $\Gamma^m$, $\Theta^m$ and $\bigvee_{i=0}^{m-1}(\Theta^i)^{-1}\mathcal{Q}$ to the inequality \eqref{ineq2}, dividing by $m$ and letting $m$ go to infinity, one has
\begin{equation*}
h_{\mu}(\Gamma\mid\mathcal{D}_{\mathcal{H}})\leq h_{\mu}(\pi_{\mathcal{E}}^{-1}\mathcal{Q}\mid \mathcal{D}_{\mathcal{H}}) +h(\Theta\mid \mathcal{Q}),
\end{equation*}
and we complete the proof.
\end{proof}

Now we can prove Theorem \ref{prop2}, which gives a variational inequality between defect of upper semi-continuity of the relative entropy function on invariant measures  and the relative tail entropy.

\begin{proof}[Proof of Theorem \ref{prop2} ]
Let $\mathcal{Q}$ be a finite  random cover of $\mathcal{E}$. By Lemma \ref{Fact666}, there exists a finite measurable partition $\mathcal{R}$ of $\mathcal{E}$ with $\mathcal{Q}\prec \mathcal{R}$ and $m(\partial R)=0$ for each $R\in \mathcal{R}$.
By Proposition \ref{prop1} and $\pi_{\mathcal{E}}\Gamma=\Theta\pi_{\mathcal{E}}$, for each $\mu\in \mathcal{I}_{\mathbb{P}}(\mathcal{H})$ and $n\in \mathbb{N}$,
\begin{align*}
 h_{\mu}(\Gamma\mid \mathcal{D}_{\mathcal{H}})
 &\leq h_{\mu}(\pi_{\mathcal{E}}^{-1}\mathcal{R}\mid\mathcal{D}_{\mathcal{H}})+h(\Theta\mid \mathcal{R})\\
 &\leq \frac{1}{n}H_{\mu}\big(
 \bigvee_{i=0}^{n-1}(\Gamma^i)^{-1}\pi_{\mathcal{E}}^{-1}\mathcal{R}\mid\mathcal{D}_{\mathcal{H}} \big)+h(\Theta\mid \mathcal{Q})\\
 \end{align*}
Then by Lemma \ref{lem2},
\begin{align*}
\limsup_{\mu\rightarrow m}h_{\mu}(\Gamma\mid \mathcal{D}_{\mathcal{H}})
&\leq \limsup_{\mu\rightarrow m}\frac{1}{n}H_{\mu}(
 \bigvee_{i=0}^{n-1}(\Gamma^i)^{-1}\pi_{\mathcal{E}}^{-1}\mathcal{R}
 \mid\mathcal{D}_{\mathcal{H}}
 )+h(\Theta\mid \mathcal{Q})\\
 &\leq \frac{1}{n}H_m(
 \bigvee_{i=0}^{n-1}(\Gamma^i)^{-1}\pi_{\mathcal{E}}^{-1}\mathcal{R}
 \mid\mathcal{D}_{\mathcal{H}}
 )+h(\Theta\mid \mathcal{Q}).
\end{align*}
Thus
\begin{equation*}
\limsup_{\mu\rightarrow m}h_{\mu}(\Gamma\mid \mathcal{D}_{\mathcal{H}})\leq h_m(\Gamma\mid \mathcal{D}_{\mathcal{H}}
 )+h(\Theta\mid \mathcal{Q}).
\end{equation*}
Since the partition $\mathcal{Q}$ is arbitrary, then
$h^*_m(\Gamma\mid\mathcal{D}_{\mathcal{H}} )\leq h^*(\Theta)$.
\end{proof}

Next we are concerned with the variational principle related with the relative entropy of $\mathcal{E}^{(2)}$ and the relative tail entropy of $\Theta$.
Recall that $\mathcal{E}^{(2)}=\{(\omega,x,y):x,y\in \mathcal{E}_{\omega}\}$ is a measurable subset of $\Omega\times X^2$ with respect to the product $\sigma-$algebra $\mathcal{F}\times \mathcal{B}^2$ and
$\mathcal{A}_{\mathcal{E}^{(2)}}=\{(A\times X)\cap\mathcal{E}^{(2)}: A\in \mathcal{F}\times \mathcal{B}\}$.
The skew product transformation $\Theta^{(2)}:\mathcal{E}^{(2)}\rightarrow \mathcal{E}^{(2)}$ is given by  $\Theta^{(2)}(\omega,x,y)=(\vartheta\omega,T_{\omega}x, T_{\omega}y)$.
Let $\mathcal{E}_1, \mathcal{E}_2$ be two copies, {\it i.e.},  $\mathcal{E}_1=\mathcal{E}_2=\mathcal{E}$, and   $\pi_{\mathcal{E}_i}$  be the natural projection from $\mathcal{E}^{(2)}$ to $\mathcal{E}_i$ with  $\pi_{\mathcal{E}_i}(\omega,x_1,x_2)=(\omega,x_i)$, i=1, 2.

\begin{proposition}\label{prop3}
Let $T$ be a continuous bundle RDS on $\mathcal{E}$ and $\mathcal{Q}=\{Q_1,\dots,Q_k\}$ be an open random  cover of $\mathcal{E}$. There exists a probability measure $\mu_{\mathcal{Q}}\in\mathcal{I}_{\mathbb{P}}(\mathcal{E}^{(2)})$ such that
\begin{enumerate}
\item[(i)]
$
 h_{\mu_{\mathcal{Q}}}( \Theta^{(2)}  \mid \mathcal{A}_{\mathcal{E}^{(2)}})\geq h(\Theta\mid \mathcal{Q})-\frac{1}{k},
$
\item[(ii)]
$\mu_{\mathcal{Q}}$ is supported on the set $\bigcup\limits_{j=1}^k\{(\omega,x,y)\in \mathcal{E}^{(2)}: x, y\in \overline{Q_j}(\omega)\}$.
\end{enumerate}
\end{proposition}

\begin{proof}
Let us choose an open random  cover $\mathcal{P}=\{P_1,\dots,P_l\}$ of $\mathcal{E}$ such that
$h_{\Theta}(\mathcal{P}\mid \mathcal{Q})\geq h(\Theta\mid \mathcal{Q})-\frac{1}{k}$.
Recall that $\mathfrak{U}(\mathcal{E})$ is the collection of all open random  covers on $\mathcal{E}$, $\mathcal{Q}^{(n)}=\bigvee_{i=0}^{n-1}(\Theta^i)^{-1}\mathcal{Q}$ and $\mathcal{Q}^{(n)}(\omega)=\bigvee_{i=0}^{n-1}(T_{\omega}^i)^{-1}\mathcal{Q}(\vartheta^i\omega)$.

Pick one element $Q(\omega)\in \mathcal{Q}^{(n)}(\omega)$ with $N(Q,\mathcal{P}^{(n)})(\omega)=N(\mathcal{P}^{(n)}\mid\mathcal{Q}^{(n)})(\omega)$ and a point $x\in Q(\omega)$.
Since $\mathcal{P}$ is an  open random cover of $\mathcal{E}$, by the compactness of $\mathcal{E}_{\omega}$, there exists a Lebesgue number $\eta(\omega)$ for the open cover $\{P_1(\omega),\dots,P_l(\omega)\}$ and a maximal $(n,\delta)-$separated subset $E_n(\omega)$ in $Q(\omega)$ such that
$$Q(\omega)\subset \bigcup_{y\in E_n(\omega)}B_y(\omega,n,\delta),$$
where $B_y(\omega,n,\delta)$ denote the open ball in $\mathcal{E}_{\omega}$ centered at $y$ of radius $1$ with respect to the metric
$d_n^{\omega}(x,y)=\max_{0\leq k < n}\{d(T_{\omega}^kx, T_{\omega}^ky)(\delta(\vartheta^k\omega))^{-1}\},$
for each $x,y \in \mathcal{E}_{\omega}$, {\it i.e.},
$B_y(\omega,n,\delta(\omega))=
\bigcap_{i=0}^{n-1}(T_{\omega}^i)^{-1}B(T^i_{\omega}y,\delta(\vartheta^i\omega))$.
Notice that for each $0\leq i\leq n-1$, the open ball $B(T^i_{\omega}y,\delta(\vartheta^i\omega))$ is contained in some element of $\mathcal{P}(\vartheta^i\omega)$, then
$B_y(\omega,n,\delta(\omega))$ must be contained in some element of $\mathcal{P}^{(n)}(\omega)$. This means that the cardinality of $E_n(\omega)$ is no less than $N(Q,\mathcal{P}^{(n)})(\omega)$.

Consider the probability measures $\sigma^{(n)}$ of $\mathcal{E}^{(2)}$ via their disintegrations
\begin{equation*}
\sigma^{(n)}_{\omega}=\frac{1}{\text{card}E_n(\omega)}\sum_{y\in E_n(\omega)}\delta_{(\omega,x,y)}
\end{equation*}
so that $d\sigma^{(n)}(\omega,x,y)=d\sigma^{(n)}_{\omega}d\mathbb{P}(\omega)$, and let
\begin{equation*}
\mu^{(n)}=\frac{1}{n}\sum_{i=0}^{n-1}(\Theta^{(2)})^i\sigma^{(n)},
\end{equation*}
 By the Krylov-Bogolyubov procedure for continuous RDS (see \cite[Theorem 1.5.8]{Arnold} or \cite[Lemma 2.1 (i)]{Kifer2001}), one can choose a subsequence $\{n_j\}$ such that $\mu^{(n_j)}$ convergence to some probability measure $\mu_{\mathcal{Q}}\in \mathcal{I}_{\mathbb{P}}(\mathcal{E}^{(2)})$.

Next we will check that the measure $\mu_{\mathcal{Q}}$ satisfies (i) and (ii).

Let $\nu=\pi_{\mathcal{E}_2}\mu_{\mathcal{Q}}$.
By Lemma \ref{Fact666}, choose a finite measurable partition $\mathcal{R}=\{R_1,\dots, R_q\}$ of $\mathcal{E}$ with $\nu(\partial R_i)=0, 0\leq i\leq q$ and $\text{diam}R_i(\omega)<\delta(\omega)$ for each $\omega$.
Set $\xi^{(n)}=\bigvee_{i=0}^{n-1}(\Theta^{(2)})^{-i}\pi_{\mathcal{E}_2}^{-1}\mathcal{R}$. Since $\pi_{\mathcal{E}_2}\Theta^{(2)}=\Theta\pi_{\mathcal{E}_2}$, then
$\xi^{(n)}=\pi_{\mathcal{E}_2}^{-1}\bigvee_{i=0}^{n-1}\Theta^{-i}\mathcal{R}
=\pi_{\mathcal{E}_2}^{-1}\mathcal{R}^{(n)}$. Denote by $\xi^{(n)}=\{D\}$ for convenience, where $D$ is a typical element of $\pi_{\mathcal{E}_2}^{-1}\mathcal{R}^{(n)}$.

For each $\omega$, let  $\pi_{X_1}^{-1}\mathcal{B}(\omega)=\{(B\times X_2)\cap
\mathcal{E}^{(2)}_{\omega}: B\in \mathcal{B}\}$, where $X_1, X_2$ are two copies of the  space $X$ and $\pi_{X_1}$ is the natural projection from the product space $X_1\times X_2$ to the space $X_1$.
We abbreviate it as $\pi_{X_1}^{-1}\mathcal{B}$ for convenience .

Since each element of $\mathcal{R}^{(n)}(\omega)$ contains at most one element of $E_n(\omega)$, one has
\begin{equation}\label{eq1}
E(1_{D(\omega)}\mid \pi^{-1}_{X_1}\mathcal{B}) (x,y)=\sigma^{(n)}_{\omega}(D(\omega)).
\end{equation}
Indeed, for each $d\in \pi_{X_1}\mathcal{B}$,
\begin{align*}
&\int_{d}E(1_{D(\omega)}\mid \pi_{X_1}\mathcal{B})d\sigma^{(n)}_{\omega}
&=\int_{d}1_{D(\omega)}d\sigma^{(n)}_{\omega}
&=\int 1_{d\cap D(\omega) }d\sigma^{(n)}_{\omega}= \sigma^{(n)}_{\omega}(d\cap D(\omega )).
\end{align*}
Since $d=(B\times X_2)\cap\mathcal{E}^{(2)}_{\omega}$ for some $B\in \mathcal{B}$ and
$D(\omega)=(X_1\times C)\cap \mathcal{E}^{(2)}_{\omega}$ for some $C\in \mathcal{R}^{(n)}(\omega)$,
Then $\sigma^{(n)}_{\omega}(d\cap D(\omega ))=\sigma^{(n)}_{\omega}((B\times C)\cap\mathcal{E}^{(2)}_{\omega} )$. By the construction of $\sigma^{(n)}_{\omega}$, one have
\begin{equation*}
\sigma^{(n)}_{\omega}((B\times X_2)\cap\mathcal{E}^{(2)}_{\omega} )=
\begin{cases}
1, & x\in B,\\
0, &\text{otherwise}.
\end{cases}
\end{equation*}
Then
\begin{align*}
\sigma^{(n)}_{\omega}((B\times C)\cap\mathcal{E}^{(2)}_{\omega} )=&\sigma^{(n)}_{\omega}((X_1\times C)\cap \mathcal{E}^{(2)}_{\omega})\cdot\sigma^{(n)}_{\omega}((B\times X_2)\cap\mathcal{E}^{(2)}_{\omega} )\\
=&\int_{(B\times X_2)\cap\mathcal{E}^{(2)}_{\omega}}\sigma^{(n)}_{\omega}((X_1\times C)\cap \mathcal{E}^{(2)}_{\omega})d\sigma^{(n)}_{\omega}\\
=&\int_d   \sigma^{(n)}_{\omega}(D(\omega))    d\sigma^{(n)}_{\omega},
\end{align*}
and so
\begin{equation*}
\int_{d}E(1_{D(\omega)}\mid \pi_{X_1}\mathcal{B})d\sigma^{(n)}_{\omega}
=\int_d   \sigma^{(n)}_{\omega}(D(\omega))    d\sigma^{(n)}_{\omega}
\end{equation*}
for all $d\in \pi_{X_1}\mathcal{B}$, which implies the  equality \eqref{eq1} holds.
Thus
\begin{align*}
H_{\sigma^{(n)}_{\omega}}(\xi^{(n)}(\omega))
&=\int\sum_{D(\omega)\in \xi^{(n)}(\omega)}-E(1_{D(\omega)}\mid \pi_{X_1}\mathcal{B})\log E(1_{D(\omega)}\mid \pi_{X_1}\mathcal{B})d\sigma^{(n)}_{\omega}\\
&=\sum_{D(\omega)\in\xi^{(n)}(\omega)}-\sigma^{(n)}_{\omega}(D(\omega))\log \sigma^{(n)}_{\omega}(D(\omega))\\
&=\log \text{card}E_n(\omega)\geq \log N(\mathcal{P}^{(n)}\mid \mathcal{Q}^{(n)})(\omega).
\end{align*}

Since for each $G\in \mathcal{A}_{\mathcal{E}^{(2)}}$,
\begin{align*}
\int_GE(1_D\mid \mathcal{A}&_{\mathcal{E}^{(2)}})d\sigma^{(n)}
=\int1_G\cdot E(1_D\mid \mathcal{A}_{\mathcal{E}^{(2)}})d\sigma^{(n)}\\
&=\int E(1_G\cdot1_D\mid \mathcal{A}_{\mathcal{E}^{(2)}})d\sigma^{(n)}\\
&=\int 1_{G\cap D}(\omega,x,y)d\sigma^{(n)}\\
&=\iint 1_{(G\cap D)(\omega)}(x,y)d\sigma^{(n)}_{\omega}d\mathbb{P}\\
&=\iint E(1_{(G\cap D)(\omega)}\mid \pi_{X_1}\mathcal{B})(x,y)d\sigma^{(n)}_{\omega}d\mathbb{P}\\
&=\iint 1_{G(\omega)}(x,y)\cdot E(1_{D(\omega)}\mid \pi_{X_1}\mathcal{B})(x,y))d\sigma^{(n)}_{\omega}d\mathbb{P}\\
&=\iint 1_{G}(\omega,x,y) E(1_{D(\omega)}\mid \pi_{X_1}\mathcal{B})(x,y))d\sigma^{(n)}_{\omega}d\mathbb{P}\\
&=\iint_{G} E(1_{D(\omega)}\mid\pi_{X_1}\mathcal{B})(x,y))d\sigma^{(n)}_{\omega}d\mathbb{P}\\
&=\int_G E(1_{D(\omega)}\mid\pi_{X_1}\mathcal{B})d\sigma^{(n)}.
\end{align*}
Then
\begin{equation*}
E(1_D\mid \mathcal{A}_{\mathcal{E}^{(2)}})(\omega,x,y)=E(1_{D(\omega)}\mid \pi_{X_1}\mathcal{B})(x,y)\,\, \mathbb{P}-a.s..
\end{equation*}
Therefore,
\begin{align}\label{ineq1}
H_{\sigma^{(n)}}&(\xi^{(n)}\mid \mathcal{A}_{\mathcal{E}^{(2)}})=
\int\sum_{D\in \xi^{(n)}}-E(1_D\mid \mathcal{A}_{\mathcal{E}^{(2)}})\log E(1_D\mid \mathcal{A}_{\mathcal{E}^{(2)}})d\sigma^{(n)}\nonumber\\
&=\iint\sum_{D(\omega)\in \xi^{(n)}(\omega)}-E(1_{D(\omega)}\mid \pi_{X_1}\mathcal{B})\log E(1_{D(\omega)}\mid \pi_{X_1}\mathcal{B})d\sigma^{(n)}_{\omega}d\mathbb{P}\nonumber\\
&=\int H_{\sigma^{(n)}_{\omega}}(\xi^{(n)}(\omega))d\mathbb{P}
\geq \int  \log N(\mathcal{P}^{(n)}\mid \mathcal{Q}^{(n)})(\omega)   d\mathbb{P}.
\end{align}

For $0\leq j<m<n$, one can cut the segment $(0,n-1)$ into disjoint union of $[\frac{n}{m}]-2$ segments $(j,j+m-1),\dots, (j+km,j+(k+1)m-1)$, $\dots$ and less than $3m$ other natural numbers. Then
\begin{align*}
H_{\sigma^{(n)}}(\xi^{(n)}\mid \mathcal{A}_{\mathcal{E}^{(2)}})
&\leq \sum_{k=0}^{[\frac{n}{m}]-2}H_{\sigma^{(n)}}
\big(
\bigvee_{i=j+km}^{j+(k+1)m-1}(\Theta^{(2)})^{-i}\pi_{\mathcal{E}_2}^{-1}\mathcal{R}\mid \mathcal{A}_{\mathcal{E}^{(2)}}
\big)
+3m\log q\\
&\leq \sum_{k=0}^{[\frac{n}{m}]-2}
H_{(\Theta^{(2)})^{j+km}\sigma^{(n)}}(\xi^{(m)}\mid \mathcal{A}_{\mathcal{E}^{(2)}})+3m\log q.
\end{align*}
By summing over all $j$, $0\leq j<m$ and considering the concavity of the entropy function $H_{(\cdot)}$, one has
\begin{align*}
mH_{\sigma^{(n)}}(\xi^{(n)}\mid \mathcal{A}_{\mathcal{E}^{(2)}})
&\leq \sum_{k=0}^{n-1}H_{(\Theta^{(2)})^k\sigma^{(n)}}(\xi^{(m)}\mid \mathcal{A}_{\mathcal{E}^{(2)}})+3m^2\log q\\
&\leq nH_{\mu^{(n)}}(\xi^{(m)}\mid \mathcal{A}_{\mathcal{E}^{(2)}})+3m^2\log q.
\end{align*}
Then by inequality \eqref{ineq1},
\begin{align*}
\frac{1}{m}H_{\mu^{(n)}}(\xi^{(m)}\mid \mathcal{A}_{\mathcal{E}^{(2)}})
&\geq
\frac{1}{n}\int   \log N(\mathcal{P}^{(n)}\mid \mathcal{Q}^{(n)})(\omega)   d\mathbb{P}-\frac{3m}{n}\log q
\end{align*}
Replacing the sequence $\{n\}$ by the above selected subsequence $\{n_j\}$ and letting $j\rightarrow\infty$, by Lemma \ref{lem2},
\begin{equation*}
\frac{1}{m}H_{\mu_{\mathcal{Q}}}(\xi^{(m)}\mid \mathcal{A}_{\mathcal{E}^{(2)}})
\geq \liminf_{j\rightarrow \infty}\frac{1}{n_j}\int \log N(\mathcal{P}^{(n_j)}\mid \mathcal{Q}^{(n_j)})(\omega)   d\mathbb{P}.
\end{equation*}
Then
$$
\frac{1}{m}H_{\mu_{\mathcal{Q}}}(\xi^{(m)}\mid \mathcal{A}_{\mathcal{E}^{(2)}})\geq h(\Theta\mid \mathcal{Q})-\frac{1}{k}.
$$
By letting $m\rightarrow \infty$, one get
$h_{\mu_{\mathcal{Q}}}(\pi_{\mathcal{E}_2}^{-1}\mathcal{R}\mid \mathcal{A}_{\mathcal{E}^{(2)}}) \geq h(\Theta\mid \mathcal{Q})-\frac{1}{k}.$

Let $\mathcal{R}_1\prec \cdots\prec \mathcal{R}_n\prec \cdots $ be an increasing sequence of finite measurable partitions with $\bigvee_{i=1}^{\infty}=\mathcal{A}$, by Lemma l.6 in \cite{Kifer} one has
$h_{\mu_{\mathcal{Q}}}(\Theta^{(2)}\mid \mathcal{A}_{\mathcal{E}^{(2)}}) \geq h(\Theta\mid \mathcal{Q}),$ which shows that the measure $\mathcal{Q}$ satisfies the property (i).

For the other part of this proposition, let $n\in\mathbb{N}$. Recall that $Q\in \mathcal{Q}^{(n)}$ and notice that $\mathcal{Q}^{(n)}\succ (\Theta^j)^{-1}\mathcal{Q}$ for all $0\leq j<n$.
Let $Q^{(2)}=\{(\omega,x,y)\in \mathcal{E}^{(2)}: x, y\in Q(\omega)\}$ and $Q_i^{(2)}=\{(\omega,x,y)\in \mathcal{E}^{(2)}: x, y\in Q_i(\omega)\}$, $1\leq i\leq k$. All of them are the measurable subsets of $\mathcal{E}^{(2)}$ with the product $\sigma-$algebra $\mathcal{F}\times \mathcal{B}^2$, and $Q^{(2)}$ is contained in $(\Theta^{(2)})^{-j}Q_i^{(2)}$ for some $1\leq i \leq k$ and $0\leq j<n$.
It follows from the construction of $\mu^{(n)}$ that
\begin{align*}
\mu^{(n)}(\bigcup_{i=1}^kQ_i^{(2)})=\frac{1}{n}\sum_{j=0}^{n-1}\sigma^{(n)}\big(
(\Theta^{(2)})^{-j}(\bigcup_{i=1}^kQ_i^{(2)})
\big)\\
\geq \frac{1}{n}\sum_{j=0}^{n-1}\sigma^{(n)}(Q^{(2)})=\sigma^{(n)}(Q^{(2)})=1.
\end{align*}
Then $$\mu^{(n)}\big(
\bigcup_{i=1}^k\{(\omega, x, y): x, y \in \overline{Q_i}(\omega)\}=1
\big).$$
Therefore the probability measure $\mu_{\mathcal{Q}}$ satisfies the property (ii) and we complete the proof.
\end{proof}

\begin{proposition}\label{prop4}
Let $T$ be a continuous bundle RDS on $\mathcal{E}$. There exists one probability measure $m\in \mathcal{I}_{\mathbb{P}}(\mathcal{E}^{(2)})$, which is supported on $\{(\omega,x,x)\in \mathcal{E}^{(2)}:x\in \mathcal{E}_{\omega}\}$, and satisfies $h^*_m(\Theta^{(2)} \mid \mathcal{A}_{\mathcal{E}^{(2)}})=h^*(\Theta)$.

\end{proposition}

\begin{proof}
Let $\mathcal{Q}_1\prec \cdots \prec \mathcal{Q}_n\prec \cdots$ be an increasing sequence of open random cover of $\mathcal{E}$. Denote by $\mathcal{Q}_n=\{Q_j^{(n)}\}_{j=1}^{k_n}$. By Property \ref{prop3}, for each $n\in \mathbb{N}$, there exists one probability measure $\mu_n\in \mathcal{I}_{\mathbb{P}}(\mathcal{E}^{(2)})$ such that $h_{\mu_n}(\Theta^{(2)}\mid \mathcal{A}_{\mathcal{E}^{(2)}})\geq h(\Theta\mid \mathcal{Q}_n)-\frac{1}{k_n}$  and $\mu_n$ is supported on $\bigcup_{j=1}^{k_n}\{(\omega,x,y):x,y\in  \overline{Q_j^{(n)}}(\omega)\}$.
Let $m$ be some limit point of the sequence of  $\mu_n$, then $m\in \mathcal{I}_{\mathbb{P}}(\mathcal{E}^{(2)})$ (see \cite[Lemma 2.1 (i)]{Kifer2001}) and
\begin{equation*}
\limsup_{\mu\rightarrow m}h_{\mu}(\Theta^{(2)}\mid \mathcal{A}_{\mathcal{E}^{(2)}})\geq \liminf_{n\rightarrow \infty}h_{\mu_n}(\Theta^{(2)}\mid \mathcal{A}_{\mathcal{E}^{(2)}})\geq \inf_n h(\Theta\mid \mathcal{Q}_n)=h^*(\Theta).
\end{equation*}

On the other hand, notice that the support of $m$
\begin{equation*}
\text{supp}m=\bigcap_{l=1}^{\infty}\bigcup_{j=1}^{k_{n_l}}\{(\omega,x,y): x, y\in \overline{Q_j^{(n_l)}}(\omega)\},
\end{equation*}
where $\{n_l\}$ is the subsequence of $\{n\}$ such that $\mu_{n_l}$ convergence to $m$ in the sense of the narrow topology.
Since $\mathcal{Q}_{n_l}$ is a refining sequence of measurable partition on $\mathcal{E}$, then
\begin{equation*}
\text{supp}m=\{(\omega,x,x)\in \mathcal{E}^{(2)}:x\in \mathcal{E}_{\omega}\}.
\end{equation*}
Thus for every finite measurable partition $\xi=\{\xi_1,\cdots,\xi_k\}$ on $\mathcal{E}$,
\begin{equation*}
m(\pi_{\mathcal{E}_1}^{-1}\xi_i)=m(\pi_{\mathcal{E}_1}^{-1}\xi_i\cap \text{supp}m)=m(\pi_{\mathcal{E}_2}^{-1}\xi_i)),\, 1\leq i\leq k,
\end{equation*}
This means $\pi_{\mathcal{E}_1}^{-1}\xi$ and $\pi_{\mathcal{E}_2}^{-1}\xi$ coincide up to sets of $m-$measure zero.
Observe that $E(1_{\pi_{\mathcal{E}_1}^{-1}\xi_i}\mid \mathcal{A}_{\mathcal{E}^{(2)}})=1_{\pi_{\mathcal{E}_1}^{-1}\xi_i}$ $\mathbb{P}-$a.s. for all $1\leq i\leq k$.
Then
\begin{equation*}
H_m(\pi^{-1}_{\mathcal{E}_2}\xi\mid \mathcal{A}_{\mathcal{E}^{(2)}})=H_m(\pi^{-1}_{\mathcal{E}_1}\xi\mid \mathcal{A}_{\mathcal{E}^{(2)}})=0,
\end{equation*}
and $h_m(\Theta^{(2)}\mid \mathcal{A}_{\mathcal{E}^{(2)}})=0$ by the definition of the relative entropy.
Hence,
\begin{equation*}
h^*_m(\Theta^{(2)}\mid \mathcal{A}_{\mathcal{E}^{(2)}})=\limsup_{\mu\rightarrow m}h_{\mu}(\Theta^{(2)}\mid \mathcal{A}_{\mathcal{E}^{(2)}})-h_m(\Theta^{(2)}\mid \mathcal{A}_{\mathcal{E}^{(2)}})\geq h^*(\Theta).
\end{equation*}
By Theorem \ref{prop2}, $h^*_m(\Theta^{(2)}\mid \mathcal{A}_{\mathcal{E}^{(2)}})\leq h^*(\Theta)$ and we complete the proof.
\end{proof}

 The  variational principle  stated in Theorem \ref{theo1} follows directly from Theorem \ref{prop2} and Proposition \ref{prop4}.

We are now in a position to prove that the relative tail entropy of a continuous bundle RDS is equal to that of its factor under the principal extension.

\begin{proof}[Proof of Theorem \ref{th3}]
Denote by
$\mathcal{G}^{(2)}=\{(\omega, y, z): y, z\in \mathcal{G}_{\omega}\}$, which is a measurable subset of $\Omega\times Y\times Y$ with respect to the product $\sigma-$algebra $\mathcal{F}\times \mathcal{C}^2$.
Let $\phi:\mathcal{G}^{(2)}\rightarrow \mathcal{E}^{(2)}$ be the map induced by the factor transformation $\pi$ as $\phi (\omega,y,z)=(\omega,\pi_{\omega}y,\pi_{\omega}z)$. Then $\phi$ is a factor transformation from $(\Lambda^{(2)},\mathcal{G}^{(2)})$ to $(\Theta^{(2)},\mathcal{E}^{(2)})$.

Let $m\in \mathcal{I}_{\mathbb{P}}(\mathcal{G}^{(2)})$ and  $\alpha:\mathcal{G}^{(2)}\rightarrow\mathcal{G}$ be the natural projection  defined as
$\alpha(\omega,y,z)=(\omega,y)$.
By the equality 4.18 in \cite{Down2011},  for each $m\in \mathcal{I}_{\mathbb{P}}(\mathcal{G}^{(2)})$,
$h_m(\Lambda^{(2)}\mid \mathcal{D}_{\mathcal{G}^{(2)}})=h_{\alpha m}(\Lambda,\mathcal{G}),$
where $h_{\alpha m}(\Lambda, \mathcal{G})$ is the usual measure-theoretical entropy.
Let $\beta:\mathcal{E}^{(2)}\rightarrow\mathcal{E}$ be the natural projection  defined as
$\beta(\omega,x,u)=(\omega,x)$.
Then $\phi m \in \mathcal{I}_{\mathbb{P}}(\mathcal{E}^{(2)})$ and  $h_{\phi m}(\Theta^{(2)}\mid \mathcal{A}_{\mathcal{E}^{(2)}})=h_{\beta(\phi m)}(\Theta, \mathcal{E}) $.

Notice that $\pi\alpha=\beta\phi$.
one obtain $h_{\beta(\phi m)}(\Theta, \mathcal{E})=h_{\pi\alpha m}(\Theta, \mathcal{E}) $.
Since the continuous bundle RDS $S$ is an principal extension of the RDS $T$ via the factor transformation $\pi$,
by the Abramov-Rokhlin formula (see \cite{BogenCrau,Liu2005})
one has $h_{\pi\alpha m}(\Theta, \mathcal{E})=h_{\alpha m}(\Lambda,\mathcal{G})$.
It follows that
$h_m(\Lambda^{(2)}\mid \mathcal{D}_{\mathcal{G}^{(2)}})=h_{\phi m}(\Theta^{(2)}\mid \mathcal{A}_{\mathcal{E}^{(2)}})$,
and then
$h^*_m(\Lambda^{(2)}\mid \mathcal{D}_{\mathcal{G}^{(2)}})=h^*_{\phi m}(\Theta^{(2)}\mid \mathcal{A}_{\mathcal{E}^{(2)}})$.
Thus by Theorem \ref{theo1},
$$h^*(\Lambda)=\max_{m\in \mathcal{I}_{\mathbb{P}}(\mathcal{G}^{(2)})}h^*_m(\Lambda^{(2)}\mid \mathcal{D}_{\mathcal{G}^{(2)}})\leq \max_{\mu\in \mathcal{I}_{\mathbb{P}}(\mathcal{E}^{(2)})}h^*_{\mu}(\Theta^{(2)}\mid \mathcal{A}_{\mathcal{E}^{(2)}})=h^*(\Theta).$$
Since for each $\mu \in \mathcal{I}_{\mathbb{P}}(\mathcal{E}^{(2)})$, by Proposition \ref{prop33}, there exists some $m\in \mathcal{I}_{\mathbb{P}}(\mathcal{G}^{(2)})$ such that $\phi m=\mu.$
Therefore the other part of the above inequality holds and we complete the proof.
\end{proof}

%

\end{document}